\documentclass[final,1p,times,authoryear]{elsarticle}
\usepackage{amsmath}
\usepackage{amsthm}
\usepackage{amssymb}
\usepackage{amsfonts}

\usepackage{mathtools}
\newcommand{\prodd}[1]{\prod_{\mathclap{#1}}}

\usepackage{csquotes}

\usepackage[capitalize]{cleveref}
\crefformat{equation}{(#2#1#3)}
\crefrangeformat{equation}{(#3#1#4) to~(#5#2#6)}
\crefmultiformat{equation}{(#2#1#3)}%
{ and~(#2#1#3)}{, (#2#1#3)}{ and~(#2#1#3)}
\crefname{lemma}{Lemma}{Lemmata}
\crefname{conjecture}{Conjecture}{Conjectures}
\usepackage{autonum}

\newtheorem{theorem}{Theorem}
\newtheorem{lemma}[theorem]{Lemma}
\newtheorem{corollary}[theorem]{Corollary}
\newtheorem{conjecture}[theorem]{Conjecture}
\newtheorem{proposition}[theorem]{Proposition}
\theoremstyle{definition}

\newtheorem{example}[theorem]{Example}

\usepackage{subcaption}
\usepackage{nicefrac}

\newcommand{\R}{\mathbb{R}}
\newcommand{\C}{\mathbb{C}}
\newcommand{\T}{\mathcal{T}}
\newcommand{\M}{\mathcal{M}}
\newcommand{\X}{\mathbb{X}}
\newcommand{\ttheta}{\boldsymbol{\theta}}
\newcommand{\ptheta}{p_{\ttheta}}
\DeclareMathOperator{\pa}{pa}
\newcommand{\treemodel}{\M_{\T}}
\newcommand{\ideal}[1]{\langle#1\rangle}

\newcommand{\ipaths}{\ensuremath{I_{\mathrm{paths}}}}
\newcommand{\impaths}{\ensuremath{I_{\mathrm{mpaths}}}}
\newcommand{\phitoric}{\varphi_{\text{\rm toric}}}
\DeclareMathOperator{\var}{V}
\newcommand{\head}[1]{{#1}_{\text{h}}}
\newcommand{\tail}[1]{{#1}_{\text{t}}}
\newcommand\independent{\protect\mathpalette{\protect\independenT}{\perp}}
\def\independenT#1#2{\mathrel{\rlap{$#1#2$}\mkern2mu{#1#2}}}

\usepackage[dvipsnames]{xcolor}

\usepackage{tikz}
\usetikzlibrary{calc}
\usetikzlibrary{arrows}
\usetikzlibrary{decorations.pathmorphing}

\newcommand{\xx}{1}
\newcommand{\yy}{1}

\newcommand{\bt}{\tikz{\node[shape=circle,draw,inner sep=2pt] {};}}
\newcommand{\ls}[1]{{\small{#1}}}

\newcommand{\stage}[2]{\tikz[baseline=(char.base)]{
            \node[shape=circle,draw,inner sep=0.5pt,fill={#1}] (char) {$v_{#2}$};}}

\begin{document}
\begin{frontmatter}
\title{Equations defining probability tree models}

\author{Eliana Duarte}
\address{Otto Von Guericke Universit\"at Magdeburg, Germany}
\ead{eliana.duarte@ovgu.de}
\ead[url]{https://emduart2.github.io}

\author{Christiane G\"orgen}
\address{Max Planck Institute for Mathematics in the Sciences, Leipzig, Germany}
\ead{goergen@mis.mpg.de}
\ead[url]{https://sites.google.com/view/goergen}

\begin{abstract} 

Staged trees or coloured probability tree models are statistical models coding conditional independence between events depicted in a tree graph. They include the very important class of Bayesian networks as a special case and provide a straightforward graphical tool for handling additional context-specific relationships. In this paper, we study the algebraic properties of their ideal of model invariants. We hereby find that the tree also provides a straightforward combinatorial tool to generalise the existing geometric characterisation of decomposable graphical models and Bayesian networks. In particular, from a staged tree we can directly understand the interplay between local and global sum-to-one conditions, read the generators of that ideal, and determine conditions under which the model is a toric variety intersected with the probability simplex.
\end{abstract}

\begin{keyword}
Algebraic Geometry; Algebraic Statistics; Graphical Models; Staged Trees
\end{keyword}
\end{frontmatter}


\section{Introduction}
The term \emph{graphical model} usually refers to a statistical model represented by a graph whose vertices are random variables and whose edges code conditional independence statements between these variables \citep{Lauritzen.1996}. Directed graphical models are also known as Bayesian networks. A more recent statistical model which can be represented graphically by a coloured probability tree is the \emph{staged tree}, sometimes called a \emph{chain event graph} \citep{Collazo.etal.2018}. Staged trees represent conditional independence relationships between depicted events rather than between random variables and do not enforce a product structure on their state space. Staged tree models are thus very general models. They contain both discrete Bayesian networks and the wider class of discrete and context-specific Bayesian networks---which allow for conditional independence statements that do not necessarily hold for all levels of a random variable \citep{Boutilier.etal.1996}---as a special case. 

Graphical statistical models are of great practical importance and their methodological foundations have now been well developed. The characterisation of so-called decomposable graphical models, and more generally of regular exponential families, as toric varieties has made an important contribution to linking properties of these models to well-known structures in algebraic geometry \citep{Pistone.etal.2001a,Geiger.etal.2006}.
However, not all graphical models are decomposable and graphical models are not always regular exponential families. Two examples of this are Bayesian networks and staged trees. These, in particular, do not necessarily correspond to toric varieties. We show in this paper that for non-decomposable graphical models the sum-to-one conditions on the parameter space cannot always be ignored, and we give sufficient conditions for staged trees to give rise to toric varieties.

In Section 2 we formally introduce staged tree models. We explain how they form  a bigger class of statistical models than discrete Bayesian networks that can treat these and their context-specific generalisations in a unified graphical framework. 
Next we focus on the combinatorial framework that these trees provide to carry out the study of the ideal of model invariants associated to any staged tree model in \cref{sect:implicit}. The elements in the ideal of model invariants are given by differences of odds ratios and they can be easily read from any tree graph representation of the model. In particular, 
this constitutes a new combinatorial approach to obtain the ideal of model invariants associated to a Bayesian network. In \cref{sect:nontoric,sect:toric}, we state a full characterisation of staged tree models using the language of commutative algebra.
 This characterisation is formally stated in terms of the algebraic and combinatorial properties of the staged tree and generalises
 previous results obtained for Bayesian networks by \citet{Garcia.etal.2005}. In \cref{thm:toric} we give sufficient \emph{graphical} conditions for a staged tree model to be a toric variety. In particular, the toric staged tree model represents a new combinatorial object to which we associate a toric ideal.  We illustrate these results in a long example developed over \cref{sect:example}, stressing how they advance the current literature on algebraic characterisations of standard graphical models. We conclude in \cref{sect:discussion} with stating a number of conjectures and outlining how our results can be interesting to researchers working in toric geometry.

Throughout the paper, we will use a number of basic notions from algebraic geometry and statistics. We refer to \citet{cox2007} for 
basics about ideals and varieties. For a more detailed treatment of toric varieties and the combinatorics of their defining
ideals we refer to \citet{Cox.etal.2011,millerSturmfels}.

\section{Staged tree models}
\label{sect:stagedtrees}

Following the formalism developed by \citet{Goergen.Smith.2017}, we always let $\T=(V,E)$ denote a directed rooted tree graph where every vertex has either no or at least two outgoing edges. The set of these outgoing edges is denoted $E(v)\subseteq E$ for every vertex $v\in V$. To every edge $e\in E$ we assign a strictly positive probability $\theta(e)\in(0,1)$ such that the labels of all edges coming out of the same vertex sum to one, $\sum_{e\in E(v)}\theta(e)=1$. We call such a labelled tree graph a \emph{probability tree}. Probability trees provide a straightforward framework to transform a problem description given in natural language into a valid statistical model \citep{Collazo.etal.2018}. They are easy to communicate to non-experts and are very versatile in applications and as representational tools in causal inference: compare \citet{Shafer.1996} and the references given at the end of this section. The positivity assumption in probability trees precludes issues related to validity assumptions as present for instance in Bayesian networks. The local sum-to-1 conditions ensure that the multiplication rule of edge labels along root-to-leaf paths in a probability tree induces a well-defined probability distribution over the graph. We denote this distribution by $\ptheta$ and the probability of a single outcome by $\ptheta(\lambda)=\prod_{e\in E(\lambda)}\theta(e)$. Here, $E(\lambda)\subseteq E$ denotes the edge set of a root-to-leaf path $\lambda$ in the tree, and the index $\ttheta=(\theta(e)~|~e\in E)$ denotes the vector of all edge labels. A \emph{probability tree model} is the set of all probability distributions which can be written in this form, for varying values of edge labels:
\begin{equation}\label{eq:treemodel}
\treemodel=\bigl\{\ptheta~|~\ttheta\in\Theta_{\T}\bigr\}.
\end{equation}
The parameter space $\Theta_{\T}$ of a probability tree model $\treemodel$ is  a product of probability simplices $\Theta_{\T}=\times_{v\in V}\Delta_{\#E(v)-1}^\circ$, one for each vertex, where 
\begin{equation}\label{eq:simplex}
\Delta_{r-1}^\circ=\Bigl\{x\in\R^r~|~\sum_{i=1}^rx_i=1\text{ and }0<x_i<1\text{ for all }i=1,\ldots,r\Bigr\}
\end{equation}
always denotes the open $r-1$-dimensional probability simplex, for some positive integer $r$. 
The probability tree model \cref{eq:treemodel} is equal to the open probability simplex $\Delta_{n-1}^\circ$ where $n$ denotes the number of root-to-leaf paths, or atomic events. 

We henceforth denote the  vector of all probabilities attached to a vertex $v\in V$ in a probability tree by the bold character $\ttheta_v=(\theta(e)~|~e\in E(v))$. A \emph{staged tree} is then a probability tree together with an equivalence relation on the vertex set such that two vertices are in the same stage if and only if their outgoing edges have the same attached probabilities: in symbols, $v\sim w$ if and only if $\ttheta_v=\ttheta_w$, possibly up to a permutation of the components of these vectors. See \cref{fig:oddsratios} on page~\pageref{fig:oddsratios} for an illustration. A \emph{staged tree model} $\treemodel$ is therefore simply a probability tree model where some probability simplices in the parameter space have been identified with each other. 
In particular, it is not the full simplex and it is the image of a parametrisation which maps a vector of parameters to a vector of products of edge labels,
\begin{equation}\label{eq:parametrisation}
\psi_{\T}:\Theta_{\T}\to\Delta_{n-1}^\circ,\quad \ttheta\mapsto \biggl(\prod_{e\in E(\lambda)}\theta(e)~|~\lambda\text{ root-to-leaf path in }\T\biggr).
\end{equation}
Even though the map \cref{eq:parametrisation} appears monomial, it is not and we cannot immediately conclude that its image is
a toric variety.
Centrally, $\psi_{\T}$ is not a map from $(\C^\ast)^{\sum_{v\in V}\#E(v)-1}$ to $(\C^\ast)^n$ because its domain is a product of simplices rather than the usual torus. One of the key contributions of this paper presented in \cref{thm:toric} is to give 
sufficient combinatorial and
algebraic conditions on the staged tree $\T$ for the image $\mathrm{im}(\psi_{\T})$ to be a toric variety. 

Staged trees can be thought of as probability trees together with conditional independence information  on the depicted events. Given that a unit in the population modelled by the tree arrives at a vertex then its immediate future unfoldings are independent of whether the unit arrived at that particular vertex or any other vertex in the same stage. In this sense, stages identify historical developments across a tree. In particular, if a tree depicts the product state space of a discrete vector of random variables, then every vertex corresponds to a random variable conditional on specific values taken by its ancestors, and an identification of probabilities attached to different vertices amounts to identifying rows of conditional probability tables. As a consequence, staged tree models include discrete and context-specific Bayesian networks as a special case \citep{Smith.Anderson.2008,Collazo.etal.2018}. In particular, staged trees are tools which can capture information about the state spaces of the underlying variables as well as \emph{all} conditional independence relationships between these purely graphically. Compare \cref{fig:radicaltree} on page~\pageref{fig:radicaltree} for an example of a staged tree which encodes conditional independence relationships that cannot be represented in a directed acyclic graph.

Because staged trees often encode a lot of symmetry, we will in \cref{sub:thmstar} introduce the concept of \emph{positions} which are stages with identical future developments. Remarkably, we find in \cref{thm:toric} and \cref{cor:simple} in that section that if all stages are positions then the parametrisation~\cref{eq:parametrisation} behaves like a monomial parametrisation. As a consequence, sum-to-$1$ conditions on the parameter space can be ignored and the resulting model is a toric variety inside a probability simplex: see also \cref{thm:implicit} in the next section.

Staged trees have now been successfully employed over a whole range of applications \citep{Freeman.Smith.2011b,Barclay.etal.2013a,Barclay.etal.2015,Collazo.Smith.2015}. \citet{Goergen.Smith.2017} characterise all different staged tree parametrisations of the same discrete model using symbolic nested representations
of a polynomial generating function.
This provides a framework to define two graphical operations which traverse the whole class of statistically equivalent staged trees, so all representations of the same model. 
An alternative approach to the question of statistical equivalence---or, in Bayesian networks, to characterise Markov equivalence---is to provide an \emph{implicit} rather than a parametric description of the model itself. All graphical representations whose corresponding distributions fulfil these implicit constraints are then statistically equivalent. We present such an algebraic characterisation and analyse its properties in this paper.

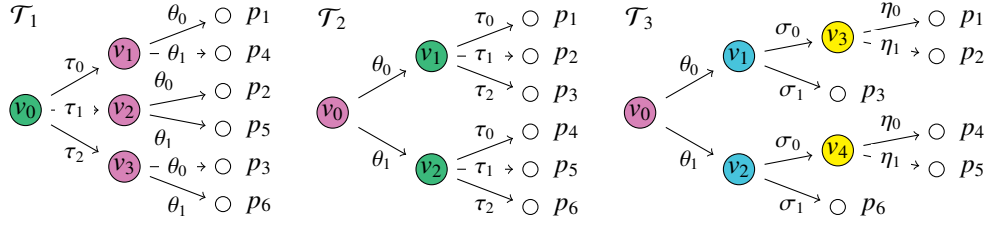
\begin{figure}
\centering
\begin{tikzpicture}
\renewcommand{\xx}{1.3}
\renewcommand{\yy}{0.5}
\node at (0,5*\yy) {$\T_1$};

\node (w0) at (0,2.5*\yy) {\stage{Green!70}{0}};
\node (w) at (1*\xx,4*\yy) {\stage{Thistle}{1}};
\node (ws) at (1*\xx,2.5*\yy) {\stage{Thistle}{2}};
\node (wu) at (1*\xx,1*\yy) {\stage{Thistle}{3}};

\node (a1) at (2*\xx,5*\yy) {\bt};
\node (a2) at (2*\xx,4*\yy) {\bt};
\node (a3) at (2*\xx,3*\yy) {\bt};
\node (a4) at (2*\xx,2*\yy) {\bt};
\node (a5) at (2*\xx,1*\yy) {\bt};
\node (a6) at (2*\xx,0*\yy) {\bt};

\draw[->] (w0) -- node [above] {\ls{$\tau_0$}} (w);
\draw[->] (w0) -- node [fill=white] {\ls{$\tau_1$}} (ws);
\draw[->] (w0) -- node [below] {\ls{$\tau_2$}} (wu);
\draw[->] (w) -- node [fill=white] {\ls{$\theta_1$}} (a2);
\draw[->] (w) -- node [above] {\ls{$\theta_0$}} (a1);
\draw[->] (ws) -- node [above, xshift=-5] {\ls{$\theta_0$}} (a3);
\draw[->] (ws) -- node [below, xshift=-5] {\ls{$\theta_1$}} (a4);
\draw[->] (wu) -- node [fill=white] {\ls{$\theta_0$}} (a5);
\draw[->] (wu) -- node [below] {\ls{$\theta_1$}} (a6);

\node [right, xshift=5] at (a1) {$p_1$};
\node [right, xshift=5] at (a2) {$p_4$};
\node [right, xshift=5] at (a3) {$p_2$};
\node [right, xshift=5] at (a4) {$p_5$};
\node [right, xshift=5] at (a5) {$p_3$};
\node [right, xshift=5] at (a6) {$p_6$};
\end{tikzpicture}%
\quad
\begin{tikzpicture}
\renewcommand{\xx}{1.3}
\renewcommand{\yy}{0.5}
\node at (0,5*\yy) {$\T_2$};

\node (v0) at (0,2.5*\yy) {\stage{Thistle}{0}};
\node (v) at (1*\xx,4*\yy) {\stage{Green!70}{1}};
\node (vf) at (1*\xx,1*\yy) {\stage{Green!70}{2}};

\node (a1) at (2*\xx,5*\yy) {\bt};
\node (a2) at (2*\xx,4*\yy) {\bt};
\node (a3) at (2*\xx,3*\yy) {\bt};
\node (a4) at (2*\xx,2*\yy) {\bt};
\node (a5) at (2*\xx,1*\yy) {\bt};
\node (a6) at (2*\xx,0*\yy) {\bt};

\draw[->] (v0) -- node [above] {\ls{$\theta_0$}} (v);
\draw[->] (v0) -- node [below] {\ls{$\theta_1$}} (vf);
\draw[->] (v) -- node [fill=white] {\ls{$\tau_1$}} (a2);
\draw[->] (v) -- node [above] {\ls{$\tau_0$}} (a1);
\draw[->] (v) -- node [below] {\ls{$\tau_2$}} (a3);
\draw[->] (vf) -- node [fill=white] {\ls{$\tau_1$}} (a5);
\draw[->] (vf) -- node [above] {\ls{$\tau_0$}} (a4);
\draw[->] (vf) -- node [below] {\ls{$\tau_2$}} (a6);

\node [right, xshift=5] at (a1) {$p_1$};
\node [right, xshift=5] at (a2) {$p_2$};
\node [right, xshift=5] at (a3) {$p_3$};
\node [right, xshift=5] at (a4) {$p_4$};
\node [right, xshift=5] at (a5) {$p_5$};
\node [right, xshift=5] at (a6) {$p_6$};
\end{tikzpicture}%
\quad
\begin{tikzpicture}
\renewcommand{\xx}{1.3}
\renewcommand{\yy}{0.5}
\node at (0,5*\yy) {$\T_3$};

\node (v0) at (0,2.5*\yy) {\stage{Thistle}{0}};
\node (v) at (1*\xx,4*\yy) {\stage{SkyBlue}{1}};
\node (vf) at (1*\xx,1*\yy) {\stage{SkyBlue}{2}};
\node (w) at (2*\xx,4.5*\yy) {\stage{Yellow}{3}};
\node (wf) at (2*\xx,1.5*\yy) {\stage{Yellow}{4}};

\node (a1) at (3*\xx,5*\yy) {\bt};
\node (a2) at (3*\xx,4*\yy) {\bt};
\node (a3) at (2*\xx,3*\yy) {\bt};
\node (a4) at (3*\xx,2*\yy) {\bt};
\node (a5) at (3*\xx,1*\yy) {\bt};
\node (a6) at (2*\xx,0*\yy) {\bt};

\draw[->] (v0) -- node [above] {\ls{$\theta_0$}} (v);
\draw[->] (v0) -- node [below] {\ls{$\theta_1$}} (vf);

\draw[->] (v) -- node [above] {\ls{$\sigma_0$}} (w);
\draw[->] (v) -- node [below] {\ls{$\sigma_1$}} (a3);
\draw[->] (vf) -- node [above] {\ls{$\sigma_0$}} (wf);
\draw[->] (vf) -- node [below] {\ls{$\sigma_1$}} (a6);

\draw[->] (w) -- node [above] {\ls{$\eta_0$}} (a1);
\draw[->] (w) -- node [fill=white] {\ls{$\eta_1$}} (a2);
\draw[->] (wf) -- node [above] {\ls{$\eta_0$}} (a4);
\draw[->] (wf) -- node [fill=white] {\ls{$\eta_1$}} (a5);

\node [right, xshift=5] at (a1) {$p_1$};
\node [right, xshift=5] at (a2) {$p_2$};
\node [right, xshift=5] at (a3) {$p_3$};
\node [right, xshift=5] at (a4) {$p_4$};
\node [right, xshift=5] at (a5) {$p_5$};
\node [right, xshift=5] at (a6) {$p_6$};
\end{tikzpicture}
\caption{Three staged probability trees where those vertices which are in the same staged are also assigned the same colour. All three of these represent the same statistical model: a discrete independence model on a binary and a ternary random variable. For $\T_{2}$, the parametrisation $
\psi_{\T_2}:\Delta_{2-1} \times \Delta_{2-1} \to \Delta_{6-1} $ is defined by 
$(\theta_0,\theta_1,\tau_0,\tau_1,\tau_2)\mapsto (\theta_0\tau_0,\theta_0\tau_1,\theta_0\tau_2,\theta_1\tau_0,\theta_1\tau_1,\theta_1\tau_2)
$. See \cref{ex:3trees} on page~\pageref{ex:3trees} for a full analysis.
} 
\label{fig:oddsratios}
\end{figure}

\section{An implicit characterisation of staged trees}\label{sect:implicit}

In this section, we employ our understanding of the statistical model represented by a staged tree to derive equations which define this model as an algebraic variety intersected with the open probability simplex. 

\subsection{Conditional probabilities and odds ratios}\label{sub:odds}

In Bayesian networks where a probability distribution factorises according to an acyclic directed graph, the atomic probabilities in the model can be written as a product of conditional probabilities depending on ancestor configurations: $p(x)=\prod_{i=1}^kp_i(x_i|x_{\pa(i)})$ where $\pa(i)$ denotes the parent set of vertex $i$ in the graph and $x$ is an atom in an underlying discrete product state space $\X_1\times\ldots\times\X_k$. These conditional probabilities locally sum to one as $\sum_{x_i\in\X_i}p_i(x_i|x_{\pa(i)})=1$ for all $i=1,\ldots,k$.
In the same fashion, we can read the distribution $\ptheta(\lambda)=\prod_{e\in E(\lambda)}\theta(e)$ over root-to-leaf paths $\lambda$ in a probability tree as a product of conditional probabilities where every label $\theta(e)$ of an edge $e=(v,v')$ denotes the transition probability of moving on to $v'$ given arrival at $v$. We show below why this is so and how this interpretation gives rise to an intuitive interpretation of stage identifications.

For clarity, we henceforth denote atomic probabilities as $p_i=\ptheta(\lambda_i)$ for root-to-leaf paths $\lambda_i$ numbered as $i=1,\ldots,n$. We write $[v]\subseteq\{1,\ldots,n\}$ for the indices of those root-to-leaf paths which pass through a fixed vertex $v\in V$ and we abbreviate the sum of their corresponding atomic probabilities to $p_{[v]}=\sum_{i\in [v]}p_i$. Thus, $p_{[v]}$ is simply the probability of the event \enquote{passing through $v$}. Whenever a staged tree is used as an equivalent representation of a Bayesian network, this type of vertex-centred event corresponds to a margin of the bigger model \citep{Collazo.etal.2018}. Paths in a tree graph which are not root-to-leaf paths are throughout denoted by their head and tail, for instance $v\to w$ for a path from $v$ to $w$. In a tree graph every such path between two vertices is unique and for our purposes we can usually ignore its directionality. The set of edges a path $v\to w$ passes through can then be denoted as $E(v\to w)\subseteq E$. The root vertex of $\T$ is always denoted by $v_0\in V$. 
\medskip

Two useful properties of $p_{[v]}$ will be employed over the following sections, for $v\in V$.
First, because of the additivity of the underlying probability measure, we can split each such probability into the sum of probabilities measured at the children of $v$:
\begin{equation}\label{eq:children}
p_{[v]}=\sum_{(v,v')\in E(v)}p_{[v']}.
\end{equation} 
This naturally follows the branching of the tree graph at each vertex and can be translated into graphical-model language as \enquote{the probability of a margin equals the sum of the probabilities of its margins}. Second, we can either recursively employ this first observation or we can directly write each atomic probability in terms of its parametrisation to find that every probability $p_{[v]}$ can be written in terms of a polynomial associated to a subtree of the bigger probability tree. To state this fact let thus $\T(v)$ denote the induced subtree of $\T$ which is rooted at $v$ and whose root-to-leaf paths correspond to $v$-to-leaf paths in $\T$. Let $w_1,\ldots,w_k$ denote the leaves of $\T(v)$, for $k\leq n$. Because $\T(v)$ is itself a probability tree, it has an associated probability distribution which for now we write as~$q$. Throughout this text, we denote by $t(v)=\sum_{j=1}^kq_{[w_j]}=\sum_{j=1}^k\prod_{e\in E(v\to w_j)}\theta(e)$ the sum of all atomic probabilities in $\T(v)$. When imposing sum-to-one conditions, by construction we have that $t(v)=1$ for all $v\in V$. This notation enables us to elegantly write
\begin{equation}\label{eq:tv}
p_{[v]}=\prod_{e\in E(v_0\to v)}\theta(e)\cdot t(v)=\prod_{e\in E(v_0\to v)}\theta(e)
\end{equation}
which implies that each $p_{[v]}$ simply equals the probability arriving at the vertex $v\in V$.

We can now easily derive the following fact.

\begin{lemma} \label{lem:condprobs}
The probability label of every edge $(v,v')\in E$ is a fraction of sums of atomic probabilities
\begin{equation}
\label{eq:condprobs}
\theta(v,v')=\dfrac{p_{[v']}}{p_{[v]}}
\end{equation}
and is the conditional probability of transitioning from $v$ to $v'$.
\end{lemma}

\begin{proof}
We simply observe that by \cref{eq:tv}, the fraction \cref{eq:condprobs} is equal to
\begin{equation}
\dfrac{\sum_{j\in[v']}p_{j}}{\sum_{i\in[v]}p_{i}}
=\dfrac{\prod_{e'\in E(v_0\to v)}\theta(e')\cdot \theta(v,v')\cdot t(v')}{\prod_{e\in E(v_0\to v)}\theta(e)\cdot t(v)}
\end{equation} 
where $t(v)$ and $t(v')$ are equal to one. These fractions are well defined because in probability trees all labels are strictly positive by definition.
\end{proof}

A straightforward implication of the above lemma is the following new result, translating a stage identification of conditional probabilities into a collection of odds-ratio equations.

\begin{lemma}\label{lem:oddsratios}
Two vertices $v,w\in V$ are in the same stage, $v\sim w$, if and only if the equation of atomic probabilities
\begin{equation}\label{eq:oddsratios}
p_{[v']}p_{[w]}=p_{[w']}p_{[v]}
\end{equation}
is true for any two edges $(v,v')\in E(v)$ and $(w,w')\in E(w)$ which share the same label.
\end{lemma}

\begin{proof}
Validity of the statement can simply be seen by noting that two vertices are in the same stage $v\sim w$ if and only if their edge labels are identified $\theta(v,v')=\theta(w,w')$ for all outgoing edges $(v,v')\in E(v)$ and $(w,w')\in E(w)$. Plugging these equations into \cref{eq:condprobs} yields the claim.
\end{proof}

An equivalent statement has been proven by \citet[Proposition~4.1.6]{Sullivant.2018} for acyclic digraphs but not for more general discrete statistical models, so in particular not for discrete context-specific conditional independence models and not for staged trees. We will fill this gap below, stating for the first time equations defining these models. We also dedicate a large part of the analysis in \cref{sect:nontoric} to the study of the exact role played by the sum-to-one conditions inherent to both the parametrisation and the associated distribution of a probability tree. These conditions are neglected by \citet{Geiger.etal.2006} who make no reference to these subtleties, but for instance \citet{Allman.Rhodes.2008} or \citet{Casanellas.Fernandez.2008} provide a successful treatment of semi-algebraic conditions in the phylogenetics literature.

Following \citet{Drton.etal.2009}, we henceforth call the ideal associated to a staged tree via the polynomials in \cref{lem:oddsratios} its \emph{ideal of model invariants}, denoted
\begin{equation}
I_{\T}=\ideal{p_{[v]}p_{[w']}-p_{[v']}p_{[w]}~|~\text{ for all }v\sim w\text{ and all }(v,v'),(w,w')\in E\text{ with the same label}}.
\end{equation}
In analogy, we call the variety $\var(I_{\T})=\{x~|~f(x)=0\text{ for all }f\in I_{\T}\}$ its \emph{model variety}. 

A direct consequence of \cref{lem:condprobs,lem:oddsratios} is the following implicit characterisation of a staged tree model as its model variety intersected with the probability simplex.

\begin{theorem}\label{thm:implicit}
Let $\treemodel$ be a staged tree model represented by a tree $\T=(V,E)$ with $n$ root-to-leaf paths and with associated ideal of model invariants $I_{\T}$. Then $\treemodel=\var(I_{\T})\cap\Delta_{n-1}^\circ$.
%
\end{theorem}

Thus staged tree models can either be directly specified using a staged tree representation---and hence an explicit parametrisation---or implicitly using the generators of the ideal of model invariants as  above. We discuss the implications of this result for statistical inference below and will then in the subsequent section move on to analysing its geometric properties.

\begin{example}\label{ex:3trees}
All of the staged trees $\T_1,\T_2,\T_3$ in \cref{fig:oddsratios} represent the same statistical model $\M_\T$. As a consequence, the closed images of all of the parametrisations $\psi_{\T_{i}}$ inside the probability simplex $\Delta_{6-1}$ are the same even though the parametrisation $\psi_{\T_{3}}$ is different from $\psi_{\T_{1}}=\psi_{\T_{2}}$.

The different staged tree representations of $\M_{\T}$ give three ideals of model invariants  with different sets of generators, namely,
\begin{equation}
\begin{split}
I_{\T_{1}}&= \ideal{p_1p_5-p_2p_4,p_1p_6-p_3p_4,p_2p_6-p_3p_5},\\
I_{\T_{2}}&=\ideal{p_{1}(p_{5}+p_{6})-p_{4}(p_{2}+p_{3}),p_{2}(p_{4}+p_{6})-p_{5}(p_{1}+p_{3}),p_{3}(p_{4}+p_{5})-p_{6}(p_{2}+p_{1})},\\
I_{\T_{3}}&= \ideal{p_1p_5-p_4p_2,p_3(p_4+p_5)-p_6(p_1+p_2)}.
\end{split}
\end{equation}
These are readily obtained using \cref{lem:oddsratios,thm:implicit}.

We can see here from the parametrisation $\psi_{\T_1}$, and hence from the staged tree $\T_1$, that $\M_{\T}$ is a toric variety: we simply read the binomial generators of $I_{\T_{1}}$ off the graph using \cref{lem:oddsratios}. From $\T_2$ and $\T_3$ it is not obvious that the model is toric, even though, of course, using \cref{thm:implicit} we can implicitly define $\M_{\T}$ as the vanishing of \emph{any} of the above ideals $I_{\T_{i}}$ intersected with the probability simplex, $i=1,2,3$.

We will see in \cref{cor:intersection} that the closed image of $\psi_{\T_{i}}$ in $\R^{6}$ is an irreducible component of all the varieties $\var(I_{\T_{i}})$.  The three respective ideals of model invariants have the decomposition
\begin{equation}
I_{\T_{1}}=\ker\varphi,\quad I_{\T_{2}}=\ker\varphi\cap \ideal{p_{4}+p_{5}+p_{6},p_{1}+p_{2}+p_{3}},\quad I_{\T_{3}}=\ker\varphi\cap \ideal{p_{1}+p_{2},p_{4}+p_{5}}.
\end{equation}
In algebraic language,
this simply means that $\ker\varphi$ is an associated minimal prime of $I_{\T_{i}}$ for $i=1,2,3$.
\end{example}

We make two key observations on \cref{thm:implicit}.

First, in statistical inference the generators of the ideal of model invariants specified in \cref{thm:implicit} is well known. A ratio of probabilities is often called an \emph{odds ratio} and identifications of odds ratios are frequently used tools in Bayesian inference. For instance, under certain conditions it can be favourable to elicit odds ratios  rather than probability distributions in order to specify a model using domain expertise \citep{Garthwaite.etal.2005}. Odds ratios naturally appear when analysing conditional independences in contingency tables \citep{Altham.1969,Altham.1970b,Altham.1970a} and models determined by this type of constraints are now well studied. Because of the correspondence of odds-ratio equations to vanishing $2\times 2$ minors of contingency tables, these results have also been of interest to the community of algebraic statistics \citep[and references therein]{Drton.etal.2009}. 

Second, in contrast to other implicit model characterisations---such as those obtained by \citet{Geiger.etal.2006} which are reviewed and extended in the following section---in the present work we give the odds-ratio equations of a staged tree model a combinatorial description that can be directly read from the tree as seen in \cref{ex:3trees}. This is the strong point of staged trees: not only do they provide a purely graphical tool to fully specify a very general statistical model but they \emph{also} provide a purely graphical tool from which we can immediately read all of the equations defining that model. This is not the case for the usual directed (and undirected) graphical models. We reiterate this point in an example developed over \cref{sect:example} which compares algebraic characterizations of different statistical models.
Furthermore, the characterization in \cref{thm:implicit} of the distributions in the interior of the probability simplex which factor according to the model can be interpreted as a version of the Hammersley-Clifford theorem for staged trees. Thus when checking whether any set of atomic probabilities factorises according to a given tree, we only need to substitute the given values into the equation specified in \cref{thm:implicit} and check whether these evaluate to zero. A discussion of this procedure can be found in \citet[Section~2.3]{Goergen.2017} and its use in model selection is discussed in \citet{Geiger.etal.2001}.

\subsection{The ideal of model invariants}\label{sub:modelideal}

The polynomial odds-ratio characterisation provided by \cref{thm:implicit} enables us to employ the language of algebraic geometry to characterise staged tree models. Decomposable graphical statistical models have in the same fashion already been successfully characterised as toric varieties \citep{Pistone.etal.2001a,Geiger.etal.2006}. Because staged tree models contain decomposable models as a special case, we can now easily extend these results and verify our advancement in a well-studied context. This section provides the foundation to do so in \cref{sect:nontoric}.
\medskip

Before proceeding into a study of their geometry, we note two properties which make staged tree models special from an algebraic viewpoint.

First, by \cref{thm:implicit} and because of the \enquote{if and only if} in \cref{lem:oddsratios}, in an implicit characterisation of staged tree models there are no inequality constraints other than those coming from the probability simplex. As a consequence, staged tree models really are algebraic varieties inside the probability simplex. This makes them pleasingly easy to handle with algebraic tools without diverting into the domain of real semi-algebraic geometry where the notion of closure might force us to handle algebraic approximations of a model rather than the model itself. The presence of hard-to-characterise inequality constraints has been a big challenge in many of the recent attempts to tackle statistical problems using algebra tools. 

Second, every staged tree model is the image of a parametrisation \cref{eq:parametrisation} whose domain is a product of probability simplices. If we extend the domain of definition of $\psi_{\T}$ to the full space $(\C^\ast)^{\sum_{v\in V}\#E(v)-1}$---retaining stage identifications but ignoring sum-to-$1$ conditions and positivity---then this map becomes monomial and its image a toric variety: compare the discussion on page~\pageref{eq:parametrisation}.

In order to be able to distinguish the cases where sum-to-1 conditions can  be ignored and where they can not, we introduce the following notation for model parametrisations.
We henceforth denote by $\R[p]=\R[p_1,\ldots,p_n]$ the polynomial ring whose indeterminates are given by atomic probabilities $p_1,\ldots,p_n$. We denote by $\R[\Theta]=\R[\theta(e)|e\in E]$ the polynomial ring whose indeterminates are given by the edge labels of a given staged tree $\T=(V,E)$.
Then the algebraic analogue of the map \cref{eq:parametrisation} is simply the ring map
\begin{equation}\label{eq:phi}
\begin{split}
\varphi :~&\R[p_{1},\ldots,p_{n}]\to \R[\Theta]/\ideal{\theta-1}\\
&p_i\mapsto \prod_{e\in E(\lambda_i)}\theta(e) \qquad\text{for all }i=1,\ldots,n
\end{split}
\end{equation}
where we use the shorthand $\ideal{\theta-1}=\ideal{\sum_{e\in E(v)}\theta(e)-1~|~v\in V}$ to denote the ideal coding the local vertex sum-to-$1$ conditions of the probability tree. These conditions imply two properties of this map. First, $\varphi$ is not a monomial map and hence $\ker\varphi$ is in general not an affine toric variety. 
Second, the polynomial $p_1+\ldots+p_n-1$ is in the kernel of $\varphi$. Thus, following common practice in algebraic statistics, we can always consider $\ker\varphi$ as a homogeneous ideal in projective space. For a longer discussion of this subtle point we refer the reader to \citet[Section~3.6]{Sullivant.2018}.

The ring map given by a probability tree parametrisation \cref{eq:parametrisation} which ignores the local sum-to-1 conditions is indeed a monomial parameterization and can simply be written as
\begin{equation}\label{eq:phitoric}
\begin{split}
\phitoric :~&\R[p_{1},\ldots,p_{n}]\to \R[\Theta]\\
&p_i\mapsto \prod_{e\in E(\lambda_i)}\theta(e) \qquad\text{for all }i=1,\ldots,n.
\end{split}
\end{equation}
The kernel of $\phitoric$ is a toric ideal.
We will analyse the relation between the kernels of the two maps \cref{eq:phi,eq:phitoric} in the following section, with a strong focus on the role played by the local sum-to-$1$ conditions. In particular, \cref{thm:toric} will give necessary graphical and algebraic conditions for the equality $\ker \varphi=\ker \phitoric$ to be true. Formulating 
 sufficient conditions for $\mathrm{im}(\psi_{\T})$ to be a toric variety  is a more subtle point in toric geometry that we
 do not address in this paper.  

Throughout the remainder of this text, we call a staged tree model toric if and only if the kernel of the associated algebraic parametrisation \cref{eq:phi} is a toric ideal. In this case, we automatically have that the kernels are equal, $\ker \varphi=\ker \phitoric$.
\Cref{fig:radicaltree} gives two examples of staged trees which are Bayesian networks that are not decomposable, one of these is clearly toric, and one example which is not a Bayesian network and not immediately seen as toric.

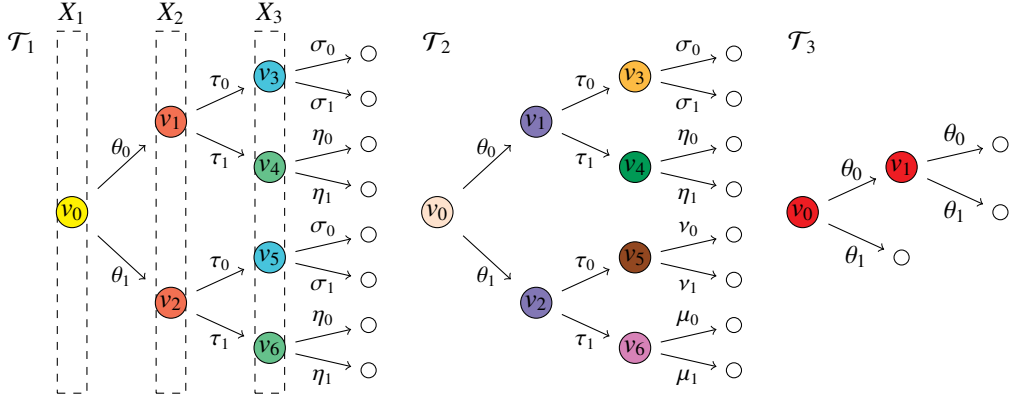
\begin{figure}
\centering
\begin{tikzpicture}
\renewcommand{\xx}{1.3}
\renewcommand{\yy}{0.3}

\node at (-0.5*\xx,15*\yy) {$\T_1$};

\draw [dashed]         (-0.15*\xx,-0.5*\yy) rectangle (0.15*\xx,15.5*\yy);
\draw [dashed]         (0.85*\xx,-0.5*\yy) rectangle (1.15*\xx,15.5*\yy);
\draw [dashed]         (1.85*\xx,-0.5*\yy) rectangle (2.15*\xx,15.5*\yy);

\node at (0*\xx,16.3*\yy) {$X_1$};
\node at (1*\xx,16.3*\yy) {$X_2$};
\node at (2*\xx,16.3*\yy) {$X_3$};
                        
\node (v0) at (0*\xx,7.5*\yy) {\stage{Yellow}{0}};
\node (v1) at (1*\xx,11.5*\yy) {\stage{Red!70}{1}};
\node (v2) at (1*\xx,3.5*\yy) {\stage{Red!70}{2}};
\node (v3) at (2*\xx,13.5*\yy) {\stage{SkyBlue}{3}};
\node (v4) at (2*\xx,9.5*\yy) {\stage{Green!60}{4}};
\node (v5) at (2*\xx,5.5*\yy) {\stage{SkyBlue}{5}};
\node (v6) at (2*\xx,1.5*\yy) {\stage{Green!60}{6}};
\node (v7) at (3*\xx,14.5*\yy) {\bt};
\node (v8) at (3*\xx,12.5*\yy) {\bt};
\node (v9) at (3*\xx,10.5*\yy) {\bt};
\node (v10) at (3*\xx,8.5*\yy) {\bt};
\node (v11) at (3*\xx,6.5*\yy) {\bt};
\node (v12) at (3*\xx,4.5*\yy) {\bt};
\node (v13) at (3*\xx,2.5*\yy) {\bt};
\node (v14) at (3*\xx,0.5*\yy) {\bt};

\draw [->]         (v0) -- node [above] {\ls{$\theta_0$}} (v1);
\draw [->]         (v0) -- node [below] {\ls{$\theta_1$}} (v2);
\draw [->]         (v1) -- node [above] {\ls{$\tau_0$}} (v3);
\draw [->]         (v1) -- node [below] {\ls{$\tau_1$}} (v4);
\draw [->]         (v2) -- node [above] {\ls{$\tau_0$}} (v5);
\draw [->]         (v2) -- node [below] {\ls{$\tau_1$}} (v6);

\draw [->]         (v3) -- node [above] {\ls{$\sigma_0$}} (v7);
\draw [->]         (v3) -- node [below] {\ls{$\sigma_1$}} (v8);
\draw [->]         (v4) -- node [above] {\ls{$\eta_0$}} (v9);
\draw [->]         (v4) -- node [below] {\ls{$\eta_1$}} (v10);
\draw [->]         (v5) -- node [above] {\ls{$\sigma_0$}} (v11);
\draw [->]         (v5) -- node [below] {\ls{$\sigma_1$}} (v12);
\draw [->]         (v6) -- node [above] {\ls{$\eta_0$}} (v13);
\draw [->]         (v6) -- node [below] {\ls{$\eta_1$}} (v14);
\end{tikzpicture}%
\quad
\begin{tikzpicture}
\renewcommand{\xx}{1.3}
\renewcommand{\yy}{0.3}

\node at (0*\xx,15*\yy) {$\T_2$};
                        
\node (v0) at (0*\xx,7.5*\yy) {\stage{Orange!20}{0}};
\node (v1) at (1*\xx,11.5*\yy) {\stage{BlueViolet!60}{1}};
\node (v2) at (1*\xx,3.5*\yy) {\stage{BlueViolet!60}{2}};
\node (v3) at (2*\xx,13.5*\yy) {\stage{Dandelion}{3}};
\node (v4) at (2*\xx,9.5*\yy) {\stage{ForestGreen}{4}};
\node (v5) at (2*\xx,5.5*\yy) {\stage{Brown!80}{5}};
\node (v6) at (2*\xx,1.5*\yy) {\stage{Thistle}{6}};
\node (v7) at (3*\xx,14.5*\yy) {\bt};
\node (v8) at (3*\xx,12.5*\yy) {\bt};
\node (v9) at (3*\xx,10.5*\yy) {\bt};
\node (v10) at (3*\xx,8.5*\yy) {\bt};
\node (v11) at (3*\xx,6.5*\yy) {\bt};
\node (v12) at (3*\xx,4.5*\yy) {\bt};
\node (v13) at (3*\xx,2.5*\yy) {\bt};
\node (v14) at (3*\xx,0.5*\yy) {\bt};

\draw [->]         (v0) -- node [above] {\ls{$\theta_0$}} (v1);
\draw [->]         (v0) -- node [below] {\ls{$\theta_1$}} (v2);
\draw [->]         (v1) -- node [above] {\ls{$\tau_0$}} (v3);
\draw [->]         (v1) -- node [below] {\ls{$\tau_1$}} (v4);
\draw [->]         (v2) -- node [above] {\ls{$\tau_0$}} (v5);
\draw [->]         (v2) -- node [below] {\ls{$\tau_1$}} (v6);

\draw [->]         (v3) -- node [above] {\ls{$\sigma_0$}} (v7);
\draw [->]         (v3) -- node [below] {\ls{$\sigma_1$}} (v8);
\draw [->]         (v4) -- node [above] {\ls{$\eta_0$}} (v9);
\draw [->]         (v4) -- node [below] {\ls{$\eta_1$}} (v10);
\draw [->]         (v5) -- node [above] {\ls{$\nu_0$}} (v11);
\draw [->]         (v5) -- node [below] {\ls{$\nu_1$}} (v12);
\draw [->]         (v6) -- node [above] {\ls{$\mu_0$}} (v13);
\draw [->]         (v6) -- node [below] {\ls{$\mu_1$}} (v14);
\end{tikzpicture}%
\quad
\begin{tikzpicture}
\renewcommand{\xx}{1.3}
\renewcommand{\yy}{0.3}

\node at (0,15*\yy) {$\T_3$};

\node (v0) at (0,7.5*\yy) {\stage{Red}{0}};
\node (v1) at (\xx,9.5*\yy) {\stage{Red}{1}};
\node (v2) at (\xx,5.5*\yy) {\bt};
\node (v3) at (2*\xx,10.5*\yy) {\bt};
\node (v4) at (2*\xx,7.5*\yy) {\bt};

\phantom{\draw (0,-0.5*\yy) -- (1,-0.5*\yy);}

\draw [->]         (v0) -- node [above] {$\theta_0$} (v1);
\draw [->]         (v0) -- node [below] {$\theta_1$} (v2);
\draw [->]         (v1) -- node [above] {$\theta_0$} (v3);
\draw [->]         (v1) -- node [below] {$\theta_1$} (v4);
\end{tikzpicture}
\caption{Three staged trees. $\T_1$ represents the conditional independence model ${X_1\independent X_2}$ and ${X_3\independent X_1~|~X_2}$ also studied in \citet{Garcia.etal.2005}. This model cannot be faithfully represented by a directed acyclic graph, it is not decomposable but toric. Indeed, $\mathcal{M}_1=\var(p_{1}p_{7}-p_{5}p_{3}, p_{1}p_{8}-p_{5}p_{4},p_{2}p_{7}-p_{6}p_{3},p_{2}p_{8}-p_{6}p_{4},p_{1}p_{6}-p_{5}p_{2},p_{3}p_{8}-p_{7}p_{4})\cap\Delta_{8-1}^\circ$. The staged tree $\T_2$ represents the non-decomposable collider Bayesian network $\mathcal{M}_2=\var((p_1+p_2)(p_7+p_8)-(p_3+p_4)(p_5+p_6))\cap\Delta_{8-1}^\circ$. The staged tree $\T_3$ represents the model $\mathcal{M}_3=\var(p_1p_3=p_2(p_1+p_2))\cap\Delta_{3-1}^\circ$ which is not a Bayesian network.}\label{fig:radicaltree}
\end{figure}

\section{Properties of staged tree models} \label{sect:nontoric}

In this section, we explore algebraic properties such as dimension of the model variety of a general staged tree. We state these precisely in the language of commutative algebra. In particular, we find that the ideal of model invariants easily links to the kernel of the associated algebraic parametrisation \cref{eq:phi} via saturation.

\subsection{Dimension of $\treemodel$}\label{sub:dim}
The dimension of $\treemodel$ specified in \cref{thm:implicit} can be inferred from the tree graph $\T$. This is because the parametrisation \cref{eq:parametrisation} between the product parameter space $\Theta_{\T}$ and the corresponding model $\treemodel$ is birational onto its image: see \cref{lem:condprobs} and compare \cref{thm:algebra} in the subsequent section. As a consequence, the dimension of $\treemodel$ is equal to the dimension of $\Theta_{\T}$, so equal to the number of free parameters in the statistical model.

In particular, assuming that the vertices of a tree graph $\T=(V,E)$ are partitioned into $r$ equivalence classes, each class corresponding to one stage (or one assignment of a colour), letting $m_i$ denote the number of vertices which are assigned the $i^\text{th}$ colour and letting $k_i$ be their number of edges, $i=1,\ldots,r$, we can write the parameter space $\Theta_{\T}$ as the product of probability simplices ${\times_{i=1}^r\Delta_{k_i-1}^\circ\subseteq\R^{\#E}}$. This space is of dimension ${\sum_{i=1}^r(k_i-1)}$. Equivalently, it can easily be seen that the total number of free parameters is equal to the number of parameters labelling the edges of the tree, minus sum-to-1 conditions and stage assignments: 
\begin{equation}
d=\#E-\#V'-\sum_{i=1}^r(m_i-1)(k_i-1)
\end{equation}
where $V'\subsetneq V$ denotes the set of non-leaf vertices of the tree. The equation $d=\sum_{i=1}^r(k_i-1)$ is true because $\#V'=\sum_{i=1}^rm_i$ and $\#E=\sum_{i=1}^rm_ik_i$.

\begin{example}
The model represented by the staged tree $\T_2$ in \cref{fig:radicaltree} is possibly not toric in the atomic probabilities. It is of dimension six, of codimension one, and its ideal of model invariants is $I_{\T_2}=\ker \varphi= \langle(p_1+p_2)(p_7+p_8)-(p_3+p_4)(p_5+p_6)\rangle $. 
In this case, the kernel of the map $\phitoric$ is empty because the atomic probabilities of this model do not satisfy any relations unless local sum-to-1 conditions are imposed on the parameters. 

However, if in $\T_2$ we colour the vertices $v_3,v_5$ blue and $v_4,v_6$ green, we obtain $\T_1$ and thus
\begin{equation}
\ker \varphi =\ker \phitoric = \langle p_{1}p_{7}-p_{5}p_{3}, p_{1}p_{8}-p_{5}p_{4},p_{2}p_{7}-p_{6}p_{3},p_{2}p_{8}-p_{6}p_{4},p_{1}p_{6}-p_{5}p_{2},p_{3}p_{8}-
  p_{7}p_{4}\rangle
\end{equation}
which is a toric ideal whose corresponding variety is of dimension four.
\end{example}

%
\subsection{Local properties of $I_{\T}$}

In \cref{thm:algebra}, the main result of this section, we state \cref{thm:implicit} in the language of commutative algebra. This result is then used to clarify the close relation between $I_\T$ and $\ker\varphi$ in terms of local rings.
\medskip

In general, we always have that the ideal of model invariants $I_{\T}\subset \ker \varphi$ is contained in the kernel of its algebraic parametrisation \cref{eq:phi}, or equivalently, that the model variety $\var(I_{\T})$ contains the closed image of $\varphi$. However, as was illustrated in \cref{ex:3trees}, equality does not always hold and different tree representations of $\T$ lead to different
primary decomposition of the ideals $I_{\T}$. 

Let henceforth $\mathbf{p}\in\R[p_1,\ldots,p_n]$ denote the product of all the denominators which appear in the identified conditional probabilities \cref{eq:condprobs} in a staged tree.

\begin{proposition}\label{thm:algebra}
For any staged tree $\T$, the localised map
\begin{equation}
\varphi_{\mathbf{p}}: \; \left(\R[p_{1},\ldots,p_{n}]/ I_{\T}\right)_{\mathbf{p}}\to \left(\R[\Theta]/\ideal{\theta-1} \right)_{\varphi(\mathbf{p})}
\end{equation}
is an isomorphism of $\R$-algebras. Hence $(\ker \varphi)_{\mathbf{p}}=(I_{\T})_{\mathbf{p}}$.
\end{proposition}

\begin{proof}
To prove that $\varphi_{\mathbf{p}}$ is an isomorphism we define the map 
\begin{equation}
\psi~:\left(\R[\Theta]/\ideal{\theta-1} \right)_{\varphi(\mathbf{p})}\to \left(\R[p_{1},\ldots,p_{n}]/ I_{\T}
\right)_{\mathbf{p}}, \quad
\theta(v,v')\mapsto \dfrac{p_{[v']}}{p_{[v]}}
\end{equation}
and check that this is an inverse for $\varphi_{\mathbf{p}}$. Let thus $\lambda=(e_{1},\ldots,e_{k})$ denote the ordered 
sequence of edges of the root-to-leaf path which is assigned atomic probability $p_{i}$. Then
\begin{equation}
\psi(\varphi_{\mathbf{p}}(p_{i}))= \psi(\prod_{i=1}^{k}\theta(e_{i}))= \prod_{i=1}^{k}\psi(\theta(e_{i})).                                   
\end{equation}
Writing the edges $e_{i}$ in terms of pairs of vertices, so $e_{1}=(v_{0},v_{1}),e_{2}=(v_{1},v_{2}),\ldots, e_{k}=(v_{k-1},v_{k})$, we see that the above can be further simplified to: 
\begin{equation}
 \prod_{i=1}^{k}\psi(\theta(e_{i})) 
 = \prod_{i=1}^{k}\frac{p_{[v_{i}]}}{p_{[v_{i-1}]}}
 = \frac{p_{[v_{1}]}}{p_{[v_{0}]}} \frac{p_{[v_{2}]}}{p_{[v_{1}]}} \cdots \frac{p_{[v_{k}]}}{p_{[v_{k-1}]}}
   = \frac{p_{[v_{k}]}}{p_{[v_{0}]}}=\frac{p_{i}}{1}=p_i
\end{equation}
for any $i=1,\ldots,n$.

For the other direction, we see that
\begin{equation}
\varphi_{\mathbf{p}}(\psi(\theta(v,v')))=\varphi_{\mathbf{p}}\left(\frac{p_{[v']}}{p_{[v]}}\right)=\theta(v,v')
\end{equation}
where the last equality follows from the \cref{lem:condprobs}. This proves the claim.
\end{proof}

The next two corollaries follow immediately from properties of ideals after localising at an element of the ring. 

\begin{corollary}\label{cor:saturation}
The kernel of $\varphi$ is the saturation of $I_{\T}$ with respect to $\mathbf{p}$, so
\begin{equation}
\ker \varphi = I_{\T}:(\mathbf{p})^{\infty}.
\end{equation}
\end{corollary}

\begin{proof}
For any ideal $I$ in a polynomial ring $R$ and any non-nilpotent element $x$, we have $I_x \cap R= I:(x)^\infty$ where $I_x$ denotes the extension of the ideal $I$
to the local ring $R_x$. Then from \cref{thm:algebra} it follows that
\begin{equation}
\ker \varphi =(\ker \varphi)_{\bf{p}}\cap R=(I_{\T})_{\bf{p}}\cap R= I_{\T}:(\bf{p})^{\infty}.
\end{equation}
\end{proof}

In \cref{thm:implicit} we saw that the ideal of model invariants $I_\T$ defines the model $\treemodel$ as a variety $\var(I_{\T})$ inside the probability simplex. \Cref{cor:intersection} states that the closed image of the parametrisation $\psi_\T$ is an irreducible component of $I_\T$.

\begin{corollary}\label{cor:intersection}
The ideal $I_{\T}$ has a primary decomposition 
\[I_{\T}=\ker \varphi \cap Q\]
where $Q$ is an intersection of primary components containing $\mathbf{p}$. In particular $\ker \varphi$ is a minimal 
associated prime of $I_{\T}$.
\end{corollary}
\begin{proof}
This follows by localizing the primary decomposition for $I_{\T}$.
\end{proof}

We remark that a result similar to \cref{cor:saturation} has been proved by \citet{Garcia.etal.2005}[Theorem 8] for 
Bayesian networks. Studying Bayesian Networks through the lense of staged trees gives a simpler proof of Theorem 8
in \citet{Garcia.etal.2005}. This new perspective  makes the relation between sum-to-one conditions in the parameter space 
and implicit description of the model more transparent and it provides an alternative combinatorial framework to study defining equations of these models. In addition, because Bayesian networks are simple special cases of staged 
trees,  \cref{thm:algebra} holds for a much more general class of models.

%
%
%

\section{Toric staged tree models}\label{sect:toric}
This section forms the core of our algebraic analysis of staged tree models. In particular, we are now ready to provide conditions under which the kernel of the ring map $\varphi$ is a toric ideal. 
In order to study these conditions, we introduce an ideal $\ipaths$ whose generators can be read from a staged tree in a way slightly different to the odds ratios. This ideal captures in a finer way the implicit equations that define $\treemodel$ and is a key ingredient to understand the case when $\ker \varphi=\ker \phitoric$. We will find in \cref{thm:toric} that these kernels are equal if and only if the local sum-to-1 conditions imposed on the domain of a probability tree parametrisation can be ignored.
\subsection{Definition and properties of $\ipaths$}
Let $v,w\in V$ and suppose the two vertices are in the same stage, so their attached labels are identified $\ttheta_v=\ttheta_w$. For simplicity of notation, let $\ttheta_{v}=(s_1,\ldots,s_k)$ be this vector of labels. Write the sets $E(v)$ and $E(w)$ of edges emanating from $v$ and $w$ as $E(v)=\{(v,v_1),\ldots,(v,v_k)\}$ and $E(w)=\{(w,w_1),\ldots,(w,w_k)\}$, respectively. Without loss of generality assume that $s_i=\theta(v,v_i)=\theta(w,w_i)$.
We define the ideal in $\R[p]$ associated to the identification $v\sim w$ as 
\begin{equation}\label{eq:stageideal}
I_{v\sim w}=\langle p_{[v_i]}p_{[w_j]}-p_{[w_i]}p_{[v_j]}~|~ i,j=1,\ldots,k\rangle
\end{equation}
and denote the sum of all of these ideals as
\begin{equation}\label{eq:ipaths}
\ipaths=\sum_{v\sim w \mbox{ {\tiny in} } \T}I_{v \sim w}.
\end{equation}
%
Recall that $p_{[v]}$ is shorthand notation for a sum of atomic probabilities, $p_{[v]}=\sum_{i\in [v]}p_i$. The generators of $\ipaths$ are quadratic polynomials that vanish on the closed image of $\psi_{\T}$. They are both algebraically and graphically closely related to the generators of $I_{\T}$ and provide an excellent tool to study the kernel of the corresponding parametrisation $\varphi$: we investigate the details of this connection in this section.

The reason we denote the ideal in \cref{eq:ipaths} as $\ipaths$ is because each generator ${p_{[v_i]}p_{[w_j]}-p_{[w_i]}p_{[v_j]}}$
can be read off the staged tree by following two paths starting and ending at the identified vertices: this is shown in \cref{pathExplain}. We thus call these generators \emph{path differences}. For simplicity, we henceforth denote the two paths coding such a path difference as the pair ${(v_{i}\to w_{j}, w_{i}\to v_{j})}$.
\medskip

Using the identity $p_{[v]}=\sum_{(v,v_i)\in E(v)}p_{[v_i]}$ derived in \cref{eq:children}, we immediately see that if $\T$ is a binary staged tree  then any odds-ratio equation $p_{[w]}p_{[v_i]}-p_{[w_i]}p_{[v]}=0$ from \cref{lem:oddsratios} which identifies two edges $(v,v_i)$ and $(w,w_i)$ is a path difference, $i=1,2$. These reduce to the unique generator of $I_{v\sim w}$. Indeed, we can explicitly calculate that $p_{[v]}=p_{[v_1]}+p_{[v_2]}$ and $p_{[w]}=p_{[w_1]}+p_{[w_2]}$ so that
\begin{equation}
\begin{split}
p_{[v]}p_{[w_1]}-p_{[v_1]}p_{[w]} &= (p_{[v_1]}+p_{[v_2]})p_{[w_1]} - p_{[v_1]}(p_{[w_1]}+p_{[w_2]})\\
&= p_{[v_2]}p_{[w_1]}-p_{[v_1]}p_{[w_2]}.
\end{split}
\end{equation}
This equality between odds-ratio differences and path differences does not hold in case $\T$ is not a binary tree: see the tree $\T_{2}$ in \cref{fig:oddsratios} for an illustration. More generally, for non-binary trees each odds-ratio difference can be written as a sum of elements in the ideal $\ipaths$. We can see this by writing the odds-ratio difference which identifies the labels $\theta(v,v_i)=\theta(w,w_i)$ as
\begin{equation}
\begin{split}
p_{[v_i]}p_{[w]}-p_{[w_i]}p_{[v]}&= p_{[v_i]}\left( \sum_{(w,w_j)\in E(w)}p_{[w_j]}\right)-p_{[w_i]}\left( \sum_{(v,v_j)\in E(v)}p_{[v_j]} \right)\\
                                 &= \sum_{j=0}^{k}p_{[v_i]}p_{[w_j]}-p_{[w_i]}p_{[v_j]}.
\end{split}
\end{equation}
The differences $p_{[v_i]}p_{[w_j]}-p_{[w_i]}p_{[v_j]}$ for $v\sim w$ and $i,j\in \{1,\ldots,k\}$ are exactly the generators of $I_{v\sim w}$. 
This discussion proves the first containment
in the next lemma.

\begin{lemma}\label{lem:subset}
$I_{\T}\subseteq \ipaths\subseteq \ker \varphi$.
\end{lemma}
\begin{proof} 

From the definition of $\ipaths$, it is enough to show that for any two vertices $v,w$ in the same stage, the generators of $I_{v\sim w}$ are in $\ker \varphi$.
Using the relation obtained in \cref{eq:tv} from \cref{sub:odds}, we have 
\begin{equation}
\varphi(p_{[v]})=t(v) \cdot\prodd{e\in E(v_0\to v)}\theta(e)= \prodd{e\in E(v_0\to v)}\theta(e).
\end{equation}
This implies that
\begin{equation}\label{eq:patheq}
\varphi(p_{[v_i]}p_{[w_j]}-p_{[w_i]}p_{[v_j]}) = \left(s_i \cdot \prodd{e\in E(v_0\to v)}\theta(e) \right)\left(s_j \cdot\prodd{e\in E(v_0\to w)}\theta(e) \right)- \left(s_j \cdot \prodd{e\in E(v_0\to w)}\theta(e) \right)\left(s_i \cdot\prodd{e\in E(v_0\to v)}\theta(e) \right)=0.
\end{equation}
\end{proof}

\begin{figure}[t]
\begin{center}
\begin{tikzpicture}
\renewcommand{\xx}{1.3}
\renewcommand{\yy}{0.3}
                        
\node (v0) at (-0.5*\xx,5.5*\yy) {\stage{White}{0}};
\node[shape=circle,draw,inner sep=2pt,outer sep=3pt,fill=Green!70] (v1) at (1*\xx,11.5*\yy) {$v$};
\coordinate (x1) at ($(v0)!0.33!(v1)$) {};
\coordinate (x2) at ($(v0)!0.66!(v1)$) {};
\node[shape=circle,draw,inner sep=1.5pt,outer sep=3pt,fill=Green!70] (v2) at (1*\xx,3.5*\yy) {$w$};

\node[shape=circle,draw,inner sep=1pt,outer sep=3pt,fill=White] (v3) at (2*\xx,13.5*\yy) {$v_{i}$};
\node[shape=circle,draw,inner sep=0.5pt,outer sep=3pt,fill=White] (v4) at (2*\xx,9.5*\yy) {$v_{j}$};
\node[shape=circle,draw,inner sep=0.5pt,outer sep=3pt,fill=White] (v5) at (2*\xx,5.5*\yy) {$w_{i}$};
\node[shape=circle,draw,inner sep=0pt,outer sep=3pt,fill=White] (v6) at (2*\xx,1.5*\yy) {$w_{j}$};

\node (v7) at (3*\xx,14.5*\yy) {};
\node (l1) at (3*\xx,13.5*\yy) {};
\node (v8) at (3*\xx,12.5*\yy) {};
\node (v9) at (3*\xx,10.5*\yy) {};
\node (v10) at (3*\xx,8.5*\yy) {};
\node (v11) at (3*\xx,6.5*\yy) {};
\node (v12) at (3*\xx,4.5*\yy) {};

\node at (2*\xx,11.5*\yy) {$\vdots$};
\node at (2*\xx,3.5*\yy) {$\vdots$};
\node at (2.75*\xx,2.75*\yy) {$\vdots$};

\begin{scope}[shift={(0.5*\xx,-3*\yy)}]
\path (3*\xx,9.5*\yy) edge[line width=5pt,color=SkyBlue] (4.1*\xx,9.5*\yy);
\path (4.25*\xx,9.5*\yy) edge[line width=5pt,color=Yellow] (5.35*\xx,9.5*\yy);
\node[right] at (3*\xx,10*\yy) {$p_{[v_i]}p_{[w_j]}=p_{[w_i]}p_{[v_j]}$};
\node at (4.25*\xx,13*\yy) {tail of path};
\node at (4.25*\xx,7*\yy) {head of path};
\path[->] (3.8*\xx,11*\yy) edge[bend left] (3.9*\xx,12*\yy);
\path[->] (5.1*\xx,11*\yy) edge[bend right] (5*\xx,12*\yy);
\path[->] (3.3*\xx,9*\yy) edge[bend right] (3.4*\xx,8*\yy);
\path[->] (4.6*\xx,9*\yy) edge[bend left] (4.5*\xx,8*\yy);
\end{scope}

\path (v0) edge[line width=4pt,color=SkyBlue,transform canvas={yshift=3pt}] (v1);
\path (v0) edge[line width=4pt,color=Yellow,transform canvas={yshift=-3pt}] (v1);
\path (v0) edge[line width=4pt,color=SkyBlue,transform canvas={yshift=-2.5pt}] (v2);
\path (v0) edge[line width=4pt,color=Yellow,transform canvas={yshift=2.5pt}] (v2);

\path (v1) edge[line width=5pt,color=SkyBlue] (v3);
\path (v1) edge[line width=5pt,color=Yellow] (v4);
\path (v2) edge[line width=5pt,color=Yellow] (v5);
\path (v2) edge[line width=5pt,color=SkyBlue] (v6);

\draw [-]         (v0) -- (x1);
\path (x1) edge[style={decorate, decoration=snake}] (x2);
\draw [->]         (x2) -- (v1);
\path[-angle 90,font=\scriptsize] (v0) edge[style={decorate, decoration=snake}] (v2);
\draw [->]         (v1) -- node [above=2pt] {$s_i$} (v3);
\draw [->]         (v1) -- node [below=2pt] {$s_j$} (v4);
\draw [->]         (v2) -- node [above=2pt] {$s_i$} (v5);
\draw [->]         (v2) -- node [below=2pt] {$s_j$} (v6);

\draw [->]         (v3) -- (v7);
\draw [->]         (v3) -- (l1);
\draw [->]         (v3) -- (v8);
\draw [->]         (v4) -- (v9);
\draw [->]         (v4) -- (v10);
\draw [->]         (v5) -- (v11);
\draw [->]         (v5) -- (v12);
\end{tikzpicture}
\end{center}
\caption{A staged tree illustrating path differences. Here, white stages can be of any colour as long as $v$ and $w$ are in the same stage. The unique path starting at $v_{i}$ and ending in $w_{j}$ then corresponds to the product $p_{[v_{i}]}p_{[w_{j}]}$ and the unique path starting at $w_{i}$ and ending at $v_{j}$ corresponds to $p_{[w_{i}]}p_{[v_{j}]}$. \label{pathExplain}}
\end{figure}
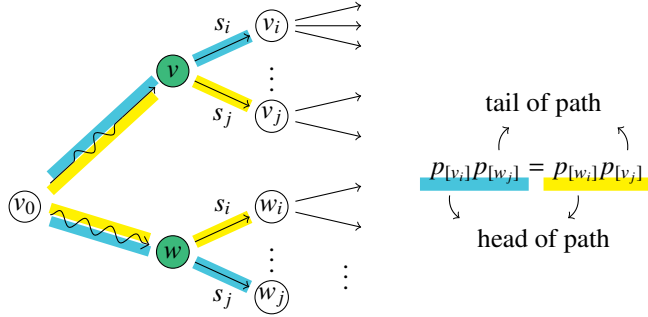

Over the next section, we will use \cref{lem:subset}, to study conditions under which the kernel $\ker \varphi$ is toric.

\subsection{Characterization of toric staged tree models} \label{sub:thmstar}
\Cref{thm:toric} presented in this section is one of the key results in this paper. It gives necessary algebraic criteria for a staged tree model to be toric. Most pleasantly, we find that such criteria can be formulated in terms of the polynomials $t(v)$ defined in \cref{sect:implicit} for $v\in V$, and can therefore be interpreted in terms of the statistical properties of the staged tree model.

To formulate this connection precisely, we briefly recall a definition from \citet{Smith.Anderson.2008}. Two vertices $v$ and $w$ in a staged tree $\T=(V,E)$ are said to be in the same \emph{position} if they are in the same stage $v\sim w$ and their induced subtrees $\T(v)$ and $\T(w)$ have the same parametrisation $\psi_{\T(v)}=\psi_{\T(w)}$. The notion of positions is of practical importance in staged tree models because it both provides a vocabulary to address vertices which have identical future unfoldings (independent of their different histories) and it provides a tool to classify subtrees representing the same statistical (sub-)model. \citet{Goergen.Smith.2017} show that two vertices $v$ and $w$ are in the same position if and only if the symbolic polynomials $t(v)$ and $t(w)$ as defined in~\cref{eq:tv} are identical: this notion is also known as \emph{polynomial equivalence}.

Consider now a path difference $p_{[v_{i}]}p_{[w_{j}]}-p_{[w_{i}]}p_{[v_{j}]}\in I_{v\sim w}$. We showed in \cref{lem:subset} that $\varphi(p_{[v_{i}]}p_{[w_{j}]}-p_{[w_{i}]}p_{[v_{j}]})=0$. In the proof of this lemma, the equation in \cref{eq:patheq} can equivalently be written as
\begin{equation} \label{eq:tpoly}
\varphi(p_{[v_i]}p_{[w_j]}-p_{[v_j]}p_{[w_i]}) 
= s_{i}s_{j}
\cdot\left(~~~\prodd{~e\in E(v_0\to v)}\theta(e)\cdot \prodd{e\in E(v_0\to w)}\theta(e)~~\right)
\cdot(t(v_{i})t(w_{j})-t(w_{i})t(v_{j})).
\end{equation}
The right hand side of \cref{eq:tpoly} can be seen as an element in $\R[\Theta]/\ideal{\theta-1}$ or an element in $\R[\Theta]$.
In the former case, due to the sum-to-1 conditions on every vertex, the polynomials associated to the respective subtrees are also equal to one, $t(v_{i})=t(w_{j})=t(w_{i})=t(v_{j})=1$, and thus 
$\varphi(p_{[v_i]}p_{[w_j]}-p_{[v_j]}p_{[w_i]}) =0$. In the latter case, the path difference does not lie in the kernel, $\varphi(p_{[v_i]}p_{[w_j]}-p_{[v_j]}p_{[w_i]})\neq 0$, unless we have $t(v_{i})t(w_{j})=t(w_{i})t(v_{j})$. This observation implies that if $t(v_{i})t(w_{j})=t(w_{i})t(v_{j})$ is
true in $\R[\Theta]$ then the path difference $p_{[v_i]}p_{[w_j]}-p_{[v_j]}p_{[w_i]}\in \ker \phitoric$ lies in the kernel of the monomial parametrisation associated to the tree. Therefore, if we assume that in the ring $\R[\Theta]$ spanned by the edge labels the equation
\begin{equation}\label{eq:star}
t(v_{i})t(w_{j})=t(w_{i})t(v_{j}) \qquad\text{ for all } i,j=1,\ldots,k 
\tag{$\star$}
\end{equation}  
is true then the ideal $I_{v\sim w}\subset \ker \phitoric $ lies in the kernel of the toric map. This leads us to the main result of this paper.

\begin{theorem}\label{thm:toric}
Let $\T$ be a staged tree. Then
\begin{equation}
\ker \phitoric =\ker \varphi
\end{equation}
if and only if condition \cref{eq:star} holds for all vertices $v,w$  of $\T$ which are in the same stage.
\end{theorem}

\begin{proof}
If condition \cref{eq:star} holds for all vertices $v,w$ in the same stage, it follows that
$\ipaths \subset \ker \phitoric $ by the definition of $\ipaths$ and the discussion before the theorem. Hence, using \cref{lem:subset}, we deduce that
$I_{\T}\subset \ker \phitoric$.
It is also straightforward to see that $\ker \phitoric \subset \ker \varphi$. We thus arrive at the chain of containments
\begin{equation}
I_{\T}\subset \ker \phitoric \subset \ker \varphi.
\end{equation}

Using \cref{thm:algebra} and localizing at $\mathbf{p}$, we see that 
\begin{equation}
(I_{\T})_{\mathbf{p}}= (\ker \phitoric )_{\mathbf{p}}= (\ker \varphi)_{\mathbf{p}}.
\end{equation}
Since both $\ker \phitoric$ and $\ker \varphi$ are prime, we obtain $\ker \phitoric =\ker \varphi$.
\end{proof}

\Cref{thm:toric} has two main implications for our algebraic characterisation of staged tree models. 

First, whenever condition \cref{eq:star} holds, the algebraic parametrisation \cref{eq:phi} behaves exactly like the monomial parametrisation \cref{eq:phitoric}. As a consequence, \cref{eq:star} \emph{is true if and only if the sum-to-1 conditions on the parameter space of the staged tree can be ignored}.

Second, we can simply read from the tree graph a sufficient condition for \cref{eq:star} to hold. In fact, if two vertices $v\sim w$ with children $v_1,v_2$ and $w_1,w_2$, respectively, are in the same stage then $t(v_{1})t(w_{2})=t(w_{1})t(v_{2})$ is satisfied  whenever $t(v_{1})=t(w_{1})$ and $t(w_{2})=t(v_{2})$. In this case, the vertices $v_1$ and $w_1$, and $v_2$ and $w_2$, are in the same position, respectively. This implies the following:

\begin{corollary}\label{cor:simple}
If in a staged tree all vertices which are in the same stage are also in the same position then the corresponding staged tree model is toric.
\end{corollary} 

In the language of \citet{Collazo.etal.2018}, \emph{simple} chain event graphs are toric.

Intuitively, coloured probability trees for which all stages are also positions have many symmetries. For example, the tree $\T_{1}$ from \cref{fig:radicaltree} satisfies condition~\cref{eq:star} for all staged vertices whereas the tree $\T_{3}$ in the same figure does not. Indeed, the kernel of the monomial parametrisation belonging to $\T_3$ is empty.
To decide wether $\T_1$ is toric we can simply check that the colored subtree rooted at $v_1$ is the same as the colored
subtree rooted at $v_2$. This example illustrates how staged trees provide a new combinatorial framework in which to study
defining equations of Bayesian networks. More precisely, although condition \cref{eq:star} is algebraic in nature it can also be checked combinatorially in terms of the colourings of the vertices of the staged tree.
\medskip

We illustrate in the next example that the result in \cref{cor:simple} is sufficient but not necessary: even if none of the vertices $v_i,v_j$ from condition~\cref{eq:star} are in the same position it could still be the case that $t(v_{i})t(w_{j})=t(w_{i})t(v_{j})$. 

\begin{example}

Consider a staged tree $\T$ where the vertices $v,w$ are the only children of the root $v_0$  and $v\sim w$. 
 Suppose $v,w$ only have two outgoing edges and denote $E(v)=\{(v,v_1),(v,v_2)\}$ and $E(w)=\{(w,w_1),(w,w_2)\}$. If 
$t(v_1)=(a_0+a_1)(b_0+b_1+b_2), t(w_2)=c_0+c_1$ and $t(w_1)=(a_0+a_1)(c_0+c_1), t(v_2)=b_0+b_1+b_2$ then the equation $t(v_{1})t(w_{2})=t(w_{1})t(v_{2})$ is satisfied but none of the vertices $v_1,v_2,w_1,w_2$ are in the same 
position.

\end{example}

\subsection{Extension of paths and binomial generators of $\ker \phitoric$}
We proved in \cref{lem:subset} that  $\ipaths$ is contained in $\ker \varphi$. In addition,  each generator of $\ipaths$ can be read from $\T$ as in \cref{pathExplain}. In this section we describe a way to extend pairs of paths associated to  generators of $\ipaths$ in such
a way that when we write down the path difference of an extended pair, we automatically get an element in $\ker \varphi$.


If we extend each path $(v_1\to w_2,w_1\to v_2)$ by one edge in such a way that the two added edges have the same parameter label, we
see that from the extended paths we can write a new path difference in $\R[p_1,\ldots,p_n]$ that is also in $\ker \varphi$. More precisely,
a pair of paths $(\head{v}\to \tail{w},\head{w}\to \tail{v})$ is said to be an \emph{extension} of $(v_1\to w_2,w_1\to v_2)$ by $l$ edges 
if the next two conditions hold. 


\begin{enumerate}
\item The path $v_h\to w_t$ is obtained from $v_1\to w_2$ by adding $l$ edges either at the head or tail of $v_1\to w_2$ and likewise for $w_h\to v_t$ and $w_1\to v_2$.
     
\item Let $\{e_1,\ldots, e_l\} = E(v_h\to w_t)\setminus E(v_1\to w_2)$ and $\{e'_1,\ldots, e'_l\}=E(w_h\to v_t)\setminus E(w_1\to v_2)$ then
\[\prod_{i=1}^{l} \theta(e_{i})=\prod_{i=1}^{l} \theta(e'_{i}).\]
\end{enumerate}

We associate the path difference $p_{[v_{h}]}p_{[w_{t}]}-p_{[w_{h}]}p_{[v_{t}]}$ to the extended pair $(v_h\to w_t,w_h\to v_t)$ and note that the second condition implies that
 $p_{[v_{h}]}p_{[w_{t}]}-p_{[w_{h}]}p_{[v_{t}]}\in \ker \varphi$.
 
 \begin{example}\label{ex:extends}
 We consider the staged tree $\T_1$ in \cref{fig:radicaltree}. We label the leaves  of $\T_{2}$\  from top to bottom by $l_1,\ldots,l_8$ and in the same fashion for the atomic probabilities $p_{1},\ldots,p_{8}$. Consider the
 pair $(v_3\to v_6,v_5\to v_4)$ and its associated path difference
 \[p_{[v_{3}]}p_{[v_{6}]}-p_{[v_{5}]}p_{[v_{4}]}=(p_{1}+p_{2})(p_{7}+p_{8})-(p_{5}+p_{6})(p_{3}+p_{4}).\]
 Notice that a possible extension of $(v_3\to v_6,v_5\to v_4)$ by one edge is given by $(l_{1}\to v_{6},l_{5}\to v_{4})$ because $\theta(v_{3},l_{1})=\theta(v_{5},l_{5})=
 \sigma_{0}$. The path difference associated to this extension is
 \[p_{[l_{1}]}p_{[v_{6}]}-p_{[l_{5}]}p_{[v_{4}]}=p_{1}(p_{7}+p_{8})-p_{5}(p_{3}+p_{4}).\] We can further extend this path by using an edge with label $\eta_{0}$ or
 $\eta_{1}$. For instance the extension $(l_{1}\to l_{7},l_{5}\to l_{3})$ with associated path difference $p_{1}p_{7}-p_{5}p_{3}$.
 
 \end{example}
  From the example above, we see that each path difference in $I_{v\sim w}$ can have several extensions to paths in $\T$. We call an extension $(v_h\to w_t,w_h\to v_t)$
  of $(v_1\to w_2,w_1\to v_2)$ a \emph{maximal extension} if it is not possible to add edges to the pair $(v_h\to w_t,w_h\to v_t)$ in such a way that condition (2) is
  satisfied. For example, the extension $(l_{1}\to l_{7},l_{5}\to l_{3})$ in \cref{ex:extends} is maximal but $(l_{1}\to v_{6},l_{5}\to v_{4})$ is not.

  We define the ideal  $\impaths$ generated by all maximal path differences  in an analogous way to $\ipaths$. Explicitly, given $v\sim w$,  we  denote
  by $I_{\mathrm{max}(v\sim w)}$ the ideal generated by
  all path differences associated to all maximal paths extending a pair $(v_{i}\to w_{j},w_{i}\to v_{j})$ for all $i,j \in \{1,\ldots, k\}$. Then 
  \begin{equation}
  \impaths= \sum_{v\sim w \mbox{ {\tiny in} } \T} I_{\mathrm{max}(v\sim w)}.
  \end{equation}

   We say that a pair 
  of paths $(v_1\to w_2,w_1\to v_2)$ \emph{fully extends} if for each maximal extension $(v_h\to w_t,w_h\to v_t)$, all of the vertices $v_h,w_t,w_h,v_t$ are leaves of $\T$. When this is the case, we see that all path differences associated to extensions are actually binomials. 
  For instance, it is easy to see that all paths in $\T_{1}$ from \cref{ex:extends} fully extend and 
    \begin{equation}
  \impaths=\langle p_{1}p_{7}-p_{5}p_{3}, p_{1}p_{8}-p_{5}p_{4},p_{2}p_{7}-p_{6}p_{3},p_{2}p_{8}-p_{6}p_{4},p_{1}p_{6}-p_{5}p_{2},p_{3}p_{8}-
  p_{7}p_{4} \rangle
  \end{equation}
  is binomial. For this ideal the first four generators correspond to extended paths and the other two are the path differences of the blue and green
  stages. Furthermore $\M_{\T_{1}}$ is toric and $\ker \varphi= \impaths$.
  \medskip
  
  Following the strategy from the previous section, we now first characterise staged trees in terms of (maximal) path differences and then use these results to understand the ideal of model invariants. 

Below, we thus start by giving a necessary condition for a path difference to fully extend. Interestingly, this condition is the same that comes up  in \cref{thm:toric} to decide wether $\ker \varphi$ is toric. In the discussion at the end of this paper, we conjecture that this is the case because for toric staged trees, the ideal $\impaths$ is equal to the kernel of $\varphi$.

\begin{lemma}\label{lem:extended}
Suppose that $v\sim w$ and $t(v_i)t(w_j)-t(w_i)t(v_j)=0$ for all $i,j \in \{1,\ldots,k\}$. Then every path $(v_i\to w_j,w_i\to v_j)$ fully extends.
\end{lemma}
\begin{proof}
For convenience, assume $i=1,j=2$. Suppose that $t(v_1)t(w_2)-t(w_1)t(v_2)=0$. For $v\in V$, each $t(v)$ is a sum of atomic probabilities of the subtree $\T(v)$. The expansion of
$t(v_1)t(w_2)$ is a sum of products of atomic probabilities in $\T(v_1),\T(w_2)$. Therefore each term in $t(v_1)t(w_2)$ cancels with a term in $t(w_1)t(v_2)$. We may write this cancelation as
$m_{v_{1}}m_{w_{2}}-m_{w_{1}}m_{v_{2}}=0$  where \[ m_{v_{1}}= \!\!\!\!\prod_{e\in E(v_{1}\to l_{1})}\!\!\!\! \theta(e),\;\; m_{w_{2}}=\!\!\!\! \prod_{e\in E(w_{2}\to l_{2})}\!\!\!\! \theta(e),\;\;
m_{w_{1}}= \!\!\!\! \prod_{e\in E(w_{1}\to l_{3})}\!\!\!\! \theta(e), \;\; m_{v_{1}}= \!\!\!\! \prod_{e\in E(v_{2}\to l_{4})}\!\!\!\! \theta(e)\]
are atomic probabilities of the subtrees $\T(v_{1}), \T(w_{2}),\T(w_{1}), \T(v_{2})$, respectively. Thus the pair
$(l_{1}\to l_{2},l_{3}\to l_{4})$ is an extension of $(v_{1}\to w_{2}, w_{1}\to v_{2})$, namely
\begin{eqnarray*}
l_{1}\to l_{2} &=& l_{1}\to v_{1}\to w_{2} \to l_{2}\\
l_{3}\to l_{4} &=& l_{3}\to v_{2}\to w_{1}\to l_{4}
\end{eqnarray*}
and $m_{v_{1}}m_{w_{2}}=m_{w_{1}}m_{v_{2}}$. The path $(l_{1}\to l_{2},l_{3}\to l_{4})$ is a full maximal extension
because $l_{1},l_{2},l_{3},l_{4}$ are leaves of $\T$. Conversely, every maximal extension is a cancellation of terms in $t(v_1)t(w_2)-t(w_1)t(v_2)$.
\end{proof}

We can now directly derive the following result, providing a sufficient condition for $\impaths$ to be binomial.

\begin{theorem}
Let $\T$ be a staged tree and suppose that condition \cref{eq:star} holds for all $v,w \in \T$ in the same stage. Then 
$\impaths$ is a binomial ideal.
\end{theorem}
\begin{proof}
Using \cref{lem:extended} we see that the path differences in $I_{v\sim w}$ fully extend for all vertices of $\T$ in the same stage.
By the definition of full extension of paths this implies that the generators of $\impaths$ are binomials.
\end{proof}



\section{An example and connections to graphical models}\label{sect:example}


In this section we illustrate the concepts introduced over the course of this paper in a simple real system: see \citet[Examples~3.4 and~3.6]{Collazo.etal.2018} for details. The discrete statistical models represented by the graphs in \cref{fig:biology} were built to explain the unfoldings of events in a cell culture. Within this culture the environment might be hostile or benign and the activity between cells might be high or low, independent of the state of the environment. If the environment is hostile then cells suffer damage and either die or survive with the same respective probabilities. Surviving cells make either a full or partial recovery, independent of their history. 

These highly asymmetric stories can be represented by the undirected graph given in \cref{fig:biodec}, the directed acyclic graph in \cref{fig:bioBN}, the staged tree in \cref{fig:biotree} or the staged tree in \cref{fig:biostaged}. We can now analyse the different statistical and algebraic properties of each of these models, contrasting our results to those obtained by \citet{Garcia.etal.2005,Geiger.etal.2006} for decomposable graphical models and Bayesian networks.
\medskip

\begin{figure}
\centering
\begin{subfigure}{\textwidth}
\centering
\begin{tikzpicture}
\renewcommand{\stage}[2]{\tikz[baseline=(char.base)]{
            \node[shape=circle,draw,inner sep=1.5pt,fill={#1}] (char) {$v_{#2}$};}}
\renewcommand{\xx}{2.5}
\renewcommand{\yy}{0.6}

\draw [dashed] 	(0.85*\xx,-0.5*\yy) rectangle  (1.15*\xx,15.5*\yy);
\draw [dashed] 	(1.85*\xx,-0.5*\yy) rectangle (2.15*\xx,15.5*\yy);
\draw [dashed] 	(2.85*\xx,-0.5*\yy) rectangle  (3.15*\xx,15.5*\yy);
\node at (0*\xx,16*\yy) {$X_1$};
\node at (1*\xx,16*\yy) {$X_2$};
\node at (2*\xx,16*\yy) {$X_3$};
\node at (3*\xx,16*\yy) {$X_4$};
			
\node (v0) at (0*\xx,7.5*\yy) {\stage{Gray}{0}};
\node (v1) at (1*\xx,11.5*\yy) {\stage{SkyBlue}{1}};
\node (v2) at (1*\xx,3.5*\yy) {\stage{SkyBlue}{2}};
\node (v3) at (2*\xx,13.5*\yy) {\stage{Green}{3}};
\node (v4) at (2*\xx,9.5*\yy) {\stage{Green}{4}};
\node (v5) at (2*\xx,5.5*\yy) {\stage{Plum!90}{5}};
\node (v6) at (2*\xx,1.5*\yy) {\stage{Thistle!80}{6}};
\node (v7) at (3*\xx,14.5*\yy) {\stage{Red}{7}};
\node (v8) at (3*\xx,12.5*\yy) {\stage{Yellow}{8}};
\node (v9) at (3*\xx,10.5*\yy) {\stage{Red}{9}};
\node[shape=circle,draw,inner sep=0.5pt,outer sep=3pt,fill=Yellow] (v10) at (3*\xx,8.5*\yy) {$v_{10}$};
\node[shape=circle,draw,inner sep=0.5pt,outer sep=3pt,fill=Red] (v11) at (3*\xx,6.5*\yy) {$v_{11}$};
\node[shape=circle,draw,inner sep=0.5pt,outer sep=3pt,fill=Yellow] (v12) at (3*\xx,4.5*\yy) {$v_{12}$};
\node[shape=circle,draw,inner sep=0.5pt,outer sep=3pt,fill=Red] (v13) at (3*\xx,2.5*\yy) {$v_{13}$};
\node[shape=circle,draw,inner sep=0.5pt,outer sep=3pt,fill=Yellow] (v14) at (3*\xx,0.5*\yy) {$v_{14}$};
\node (v15) at (4*\xx,15*\yy) {\bt};
\node (v16) at (4*\xx,14*\yy) {\bt};
\node (v17) at (4*\xx,13*\yy) {\bt};
\node (v18) at (4*\xx,12*\yy) {\bt};
\node (v19) at (4*\xx,11*\yy) {\bt};
\node (v20) at (4*\xx,10*\yy) {\bt};
\node (v21) at (4*\xx,9*\yy) {\bt};
\node (v22) at (4*\xx,8*\yy) {\bt};
\node (v23) at (4*\xx,7*\yy) {\bt};
\node (v24) at (4*\xx,6*\yy) {\bt};
\node (v25) at (4*\xx,5*\yy) {\bt};
\node (v26) at (4*\xx,4*\yy) {\bt};
\node (v27) at (4*\xx,3*\yy) {\bt};
\node (v28) at (4*\xx,2*\yy) {\bt};
\node (v29) at (4*\xx,1*\yy) {\bt};
\node (v30) at (4*\xx,0*\yy) {\bt};

\draw [->] 	(v0) -- node [above, sloped] {\ls{hostile}} node [below,sloped] {\ls{$X_1=0$}} (v1);
\draw [->] 	(v0) -- node [below, sloped] {\ls{benign}} node [above,sloped] {\ls{$X_1=1$}} (v2);
\draw [->] 	(v1) -- node [above, sloped] {\ls{high}} node [below,sloped] {\ls{$X_2=0$}} (v3);
\draw [->] 	(v1) -- node [below, sloped] {\ls{low}} node [above,sloped] {\ls{$X_2=1$}} (v4);
\draw [->] 	(v2) -- node [above, sloped] {\ls{high}} node [below,sloped] {\ls{$X_2=0$}} (v5);
\draw [->] 	(v2) -- node [below, sloped] {\ls{low}}  node [above,sloped] {\ls{$X_2=1$}}(v6);

\draw [->] 	(v3) -- node [above, sloped] {\ls{die}} node [below=-2pt,sloped] {\ls{$X_3=0$}} (v7);
\draw [->] 	(v3) -- node [below, sloped] {\ls{survive}} node [above=-3pt,sloped,pos=0.7] {\ls{$X_3=1$}} (v8);
\draw [->] 	(v4) -- node [above, sloped] {\ls{die}} node [below=-2pt,sloped] {\ls{$X_3=0$}} (v9);
\draw [->] 	(v4) -- node [below, sloped] {\ls{survive}} node [above=-3pt,sloped,pos=0.7] {\ls{$X_3=1$}} (v10);
\draw [->] 	(v5) -- node [above, sloped] {\ls{die}} node [below=-2pt,sloped] {\ls{$X_3=0$}} (v11);
\draw [->] 	(v5) -- node [below, sloped] {\ls{survive}} node [above=-3pt,sloped,pos=0.7] {\ls{$X_3=1$}} (v12);
\draw [->] 	(v6) -- node [above, sloped] {\ls{die}} node [below=-2pt,sloped] {\ls{$X_3=0$}} (v13);
\draw [->] 	(v6) -- node [below, sloped] {\ls{survive}} node [above=-3pt,sloped,pos=0.7] {\ls{$X_3=1$}} (v14);

\draw [->] 	(v7) -- node [above=-2pt, sloped] {\ls{full $X_4=0$}} (v15);
\draw [->] 	(v7) -- node [below=-2pt, sloped] {\ls{partial}} node [above=-3pt,sloped,pos=0.7] {\ls{$X_4=1$}} (v16);
\draw [->] 	(v8) -- node [above=-2pt, sloped] {\ls{full $X_4=0$}} (v17);
\draw [->] 	(v8) -- node [below=-2pt, sloped] {\ls{partial}} node [above=-3pt,sloped,pos=0.7] {\ls{$X_4=1$}} (v18);
\draw [->] 	(v9) -- node [above=-2pt, sloped] {\ls{full $X_4=0$}} (v19);
\draw [->] 	(v9) -- node [below=-2pt, sloped] {\ls{partial}} node [above=-3pt,sloped,pos=0.7] {\ls{$X_4=1$}} (v20);
\draw [->] 	(v10) -- node [above=-2pt, sloped] {\ls{full $X_4=0$}} (v21);
\draw [->] 	(v10) -- node [below=-2pt, sloped] {\ls{partial}} node [above=-3pt,sloped,pos=0.7] {\ls{$X_4=1$}} (v22);
\draw [->] 	(v11) -- node [above=-2pt, sloped] {\ls{full $X_4=0$}} (v23);
\draw [->] 	(v11) -- node [below=-2pt, sloped] {\ls{partial}} node [above=-3pt,sloped,pos=0.7] {\ls{$X_4=1$}} (v24);
\draw [->] 	(v12) -- node [above=-2pt, sloped] {\ls{full $X_4=0$}} (v25);
\draw [->] 	(v12) -- node [below=-2pt, sloped] {\ls{partial}} node [above=-3pt,sloped,pos=0.7] {\ls{$X_4=1$}} (v26);
\draw [->] 	(v13) -- node [above=-2pt, sloped] {\ls{full $X_4=0$}} (v27);
\draw [->] 	(v13) -- node [below=-2pt, sloped] {\ls{partial}} node [above=-3pt,sloped,pos=0.7] {\ls{$X_4=1$}} (v28);
\draw [->] 	(v14) -- node [above=-2pt, sloped] {\ls{full $X_4=0$}} (v29);
\draw [->] 	(v14) -- node [below=-2pt, sloped] {\ls{partial}} node [above=-3pt,sloped,pos=0.7] {\ls{$X_4=1$}} (v30);

\node at (4.4*\xx,15*\yy) {$p_{0000}$};
\node at (4.4*\xx,14*\yy) {$p_{0001}$};
\node at (4.4*\xx,13*\yy) {$p_{0010}$};
\node at (4.4*\xx,12*\yy) {$p_{0011}$};
\node at (4.4*\xx,11*\yy) {$p_{0100}$};
\node at (4.4*\xx,10*\yy) {$p_{0101}$};
\node at (4.4*\xx,9*\yy) {$p_{0110}$};
\node at (4.4*\xx,8*\yy) {$p_{0111}$};
\node at (4.4*\xx,7*\yy) {$p_{1000}$};
\node at (4.4*\xx,6*\yy) {$p_{1001}$};
\node at (4.4*\xx,5*\yy) {$p_{1010}$};
\node at (4.4*\xx,4*\yy) {$p_{1011}$};
\node at (4.4*\xx,3*\yy) {$p_{1100}$};
\node at (4.4*\xx,2*\yy) {$p_{1101}$};
\node at (4.4*\xx,1*\yy) {$p_{1110}$};
\node at (4.4*\xx,0*\yy) {$p_{1111}$};
\end{tikzpicture}
\caption{A staged tree $\T$ representing a context-specific Bayesian network.\label{fig:biotree}}
\end{subfigure}

\begin{subfigure}[b]{0.28\textwidth}
\begin{tikzpicture}
\renewcommand{\yy}{0.7}
\renewcommand{\xx}{1.15}
\node (x1) at (0,\yy) {$X_1$};
\node[above=4pt] at (x1) {\ls{environment}};
\node (x2) at (0,-\yy) {$X_2$};
\node[below=4pt] at (x2) {\ls{activity}};
\node (x3) at (\xx,0) {$X_3$};
\node[below=5pt] at (x3) {\ls{survival}};
\node (x4) at (2*\xx,0) {$X_4$};
\node[above=4pt] at (x4) {\ls{recovery}};
\draw[-]    (x1) -- (x2);
\draw[-] 	(x1) -- (x3);
\draw[-]	(x2) -- (x3);
\draw[-]	(x3) -- (x4); 
\node at (0,-3*\yy) {~};
\end{tikzpicture}
\caption{A decomposable undirected graphical model. \label{fig:biodec}}
\end{subfigure}~
\begin{subfigure}[b]{0.28\textwidth}
\centering
\begin{tikzpicture}
\renewcommand{\yy}{0.7}
\renewcommand{\xx}{1.15}
\node (x1) at (0,\yy) {$X_1$};
\node[above=4pt] at (x1) {\ls{environment}};
\node (x2) at (0,-\yy) {$X_2$};
\node[below=4pt] at (x2) {\ls{activity}};
\node (x3) at (\xx,0) {$X_3$};
\node[below=5pt] at (x3) {\ls{survival}};
\node (x4) at (2*\xx,0) {$X_4$};
\node[above=4pt] at (x4) {\ls{recovery}};
\draw[->] 	(x1) -- (x3);
\draw[->]	(x2) -- (x3);
\draw[->]	(x3) -- (x4); 
\node at (0,-3*\yy) {~};
\end{tikzpicture}
\caption{A Bayesian network that is not decomposable. \label{fig:bioBN}}
\end{subfigure}~
\begin{subfigure}[b]{0.4\textwidth}
\centering
\begin{tikzpicture}
\renewcommand{\stage}[2]{\tikz[baseline=(char.base)]{
            \node[shape=circle,draw,inner sep=1.5pt,fill={#1}] (char) {$v_{#2}$};}}
\renewcommand{\xx}{1.25}
\renewcommand{\yy}{0.8}

\node (v0) at (0*\xx,2*\yy) {\stage{Gray}{0}};
\node (v1) at (1*\xx,3*\yy) {\stage{SkyBlue}{1}};
\node (v2) at (1*\xx,1*\yy) {\stage{SkyBlue}{2}};
\node (v3) at (2*\xx,4*\yy) {\stage{Green}{3}};
\node (v4) at (2*\xx,2.5*\yy) {\stage{Green}{4}};
\node (v5) at (2*\xx,1*\yy) {\bt};
\node (v6) at (2*\xx,0*\yy) {\bt};
\node (v7) at (3*\xx,5*\yy) {\bt};
\node (v8) at (3*\xx,4*\yy) {\stage{Yellow}{8}};
\node (v9) at (3*\xx,3*\yy) {\bt};
\node[shape=circle,draw,inner sep=0.5pt,outer sep=3pt,fill=Yellow] (v10) at (3*\xx,2*\yy) {$v_{10}$};
\node (v11) at (4*\xx,4.5*\yy) {\bt};
\node (v12) at (4*\xx,3.5*\yy) {\bt};
\node (v13) at (4*\xx,2*\yy) {\bt};
\node (v14) at (4*\xx,1*\yy) {\bt};

\draw [->] 	(v0) -- node [above, sloped] {\ls{hostile}} (v1);
\draw [->] 	(v0) -- node [below, sloped] {\ls{benign}} (v2);
\draw [->] 	(v1) -- node [above, sloped] {\ls{high}} (v3);
\draw [->] 	(v1) -- node [below, sloped] {\ls{low}} (v4);
\draw [->] 	(v2) -- node [above, sloped] {\ls{high}} (v5);
\draw [->] 	(v2) -- node [below, sloped] {\ls{low}} (v6);

\draw [->] 	(v3) -- node [above, sloped] {\ls{die}} (v7);
\draw [->] 	(v3) -- node [below, sloped] {\ls{survive}} (v8);
\draw [->] 	(v4) -- node [above, sloped] {\ls{die}} (v9);
\draw [->] 	(v4) -- node [below, sloped] {\ls{survive}} (v10);
\draw [->] 	(v8) -- node [above, sloped] {\ls{full}} (v11);
\draw [->] 	(v8) -- node [below, sloped] {\ls{partial}} (v12);
\draw [->] 	(v10) -- node [above, sloped] {\ls{full}} (v13);
\draw [->] 	(v10) -- node [below, sloped] {\ls{partial}} (v14);
\end{tikzpicture}
\caption{A staged tree $\T_{>0}$ which is not a context-specific Bayesian network. \label{fig:biostaged}}
\end{subfigure}
\caption{Four graphical models for the unfoldings of events in a cell culture.\label{fig:biology}}
\end{figure}
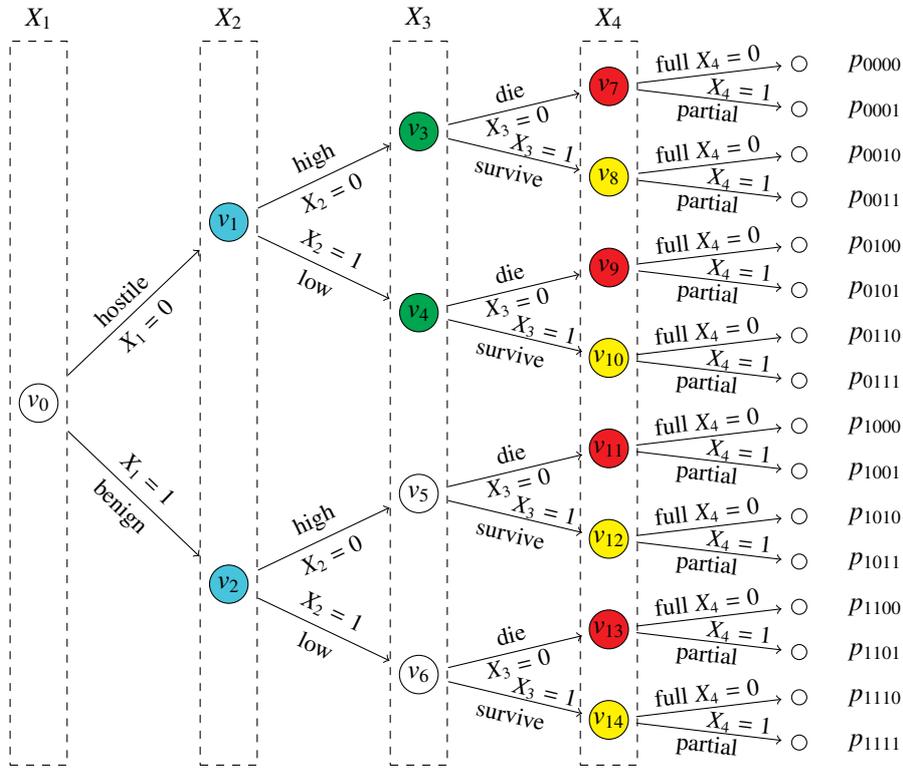
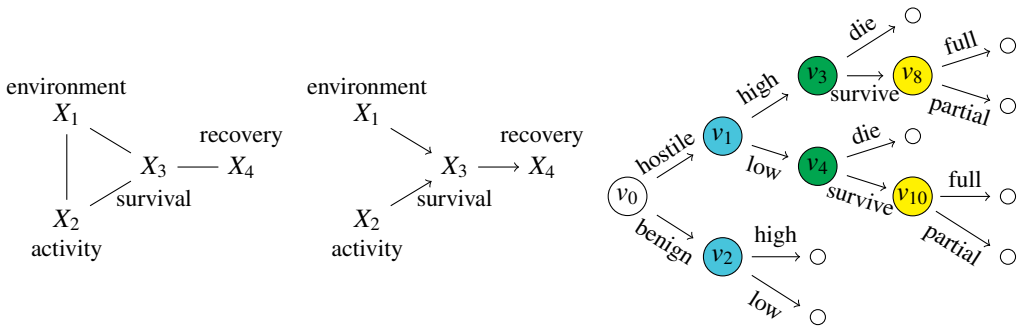

Consider first the decomposable model from \cref{fig:biodec}. This graph represents the single conditional independence assumption that recovery of a cell is independent of its activity and of the state of the environment, given survival. In symbols, ${X_4\independent (X_1,X_2)~|~X_3}$. As a staged tree, this model can be represented by the graph in \cref{fig:biotree} with the vertices $v_{7},v_{9},v_{11},v_{13}$ and $v_{8},v_{10},v_{12},v_{14}$ in the same stage, respectively. So rather than working on the undirected graph, we can equivalently exclusively consider the red and yellow colouring of this tree, for the moment ignoring the other colours: for simplicity, we denote this staged tree as $\T_{\text{dec}}$. The ideal of model invariants of $\T_{\text{dec}}$ is given by
\begin{equation}
I_{\T_{\text{dec}}}=I_{\text{red}}+I_{\text{yellow}}
\end{equation}
where the generators of the stage ideals $I_{\text{red}}={I_{v_7\sim v_9}+I_{v_9\sim v_{11}}+I_{v_{11}\sim v_{13}}}$ and $I_{\text{yellow}}=I_{v_8\sim v_{10}}+I_{v_{10}\sim v_{12}}+I_{v_{12}\sim v_{14}}$ are odds-ratio differences which can be read from the tree as
\begin{equation}
\begin{split}
I_{\text{red}}=\ideal{&p_{1001}p_{1100} - p_{1000}p_{1101},
p_{0101}p_{1100} - p_{0100}p_{1101},
p_{0001}p_{1100} - p_{0000}p_{1101},\\
&p_{0101}p_{1000} - p_{0100}p_{1001},
p_{0001}p_{1000} - p_{0000}p_{1001},
p_{0001}p_{0100} - p_{0000}p_{0101}}\\
I_{\text{yellow}}=\ideal{&p_{1011}p_{1110} - p_{1010}p_{1111},
p_{0111}p_{1110} - p_{0110}p_{1111},
p_{0011}p_{1110} - p_{0010}p_{1111},\\
&p_{0111}p_{1010} - p_{0110}p_{1011},
p_{0011}p_{1010} - p_{0010}p_{1011},
p_{0011}p_{0110} - p_{0010}p_{0111}}.
\end{split}
\end{equation}

Because $\T_{\text{dec}}$ is a binary tree, the ideal $I_{\T_{\text{dec}}}=\ipaths$ is equal to the ideal generated by path differences. By construction, $I_{\T_{\text{dec}}}$ is also equal to the ideal $I_{\text{local}(G)}$ generated by cross-product differences coding the local Markov property in the decomposable graph $G$ in \cref{fig:biodec}, as defined by \citet{Geiger.etal.2006}. In particular, the model is thus equal to the variety $\var(I_{\T_{\text{dec}}})=\var(I_{\text{local}(G)})$ intersected with the probability simplex. To illustrate the direct translation of our methods into their framework, here we use notation used by the authors cited above and denote by $p_{ijkl}=P(X_1=i,X_2=j,X_3=k,X_4=l)$ an atomic probability and by $p_{ij++}=\sum_{kl}p_{ijkl}$ a marginal probability, for $i,j,k,l=0,1$. Naturally, the atomic probabilities are attached to leaves of the tree, so $p_{ijkl}$ is the probability of the root-to-leaf path labelled $X_1=i,X_2=j,X_3=k,X_4=l$. The marginal probabilities can equivalently clearly be expressed as probabilities of vertex-centred events in the tree graph: for instance, $p_{00++}=p_{[v_3]}=\theta(v_0,v_1)\theta(v_1,v_3)\cdot t(v_3)$ where the labels $\theta(v_0,v_1)=P(X_1=0)$ and $\theta(v_1,v_3)=P(X_2=0|X_1=0)$ are conditional probabilities and $t(v_3)$ is the sum of atomic probabilities in the induced subtree $\T(v_3)$.

We observe that the ideal $I_{\text{red}}+I_{\text{yellow}}$ is toric because the path differences coding the red and the yellow stage belong to pairs of paths which extend to the leaves of the tree simply because the children of these stages are leaves. Using any computer algebra software, we can check that the ideal of model invariants of the staged tree is also equal to the kernel of the algebraic parametrisation: $I_{\T_{\text{dec}}}=\ker\varphi$.

Equivalently, \citet{Geiger.etal.2006} prove in their Theorem~4.3 that $I_{\T_{\text{dec}}}=I_{\text{local}(G)}$ is toric because it is the ideal of model invariants of a decomposable graphical model. We can see here that, in contrast to the staged tree, the generators of this ideal cannot be directly read from the decomposable graph but need to be calculated via equations associated to the saturated conditional independence statement ${X_4\independent (X_1,X_2)~|~X_3}$. This can be achieved for instance using the \texttt{Macaulay2}-package \texttt{GraphicalModels} \citep{GraphicalModelsSource}.
\medskip

Consider now the Bayesian network given in \cref{fig:bioBN} which is not decomposable. This graph codes the additional condition that the state of the environment and the activity within a cell culture are independent of each other: so here, $X_1\independent X_2$ as well as ${X_4\independent (X_1,X_2)~|~X_3}$. This model can alternatively be represented by a staged tree denoted $\T_{\text{BN}}$ which has the same coloured graph as $\T_{\text{dec}}$ but with an additional stage $v_1\sim v_2$, coloured blue. The ideal of model invariants of this tree is again generated by path differences and is equal to $I_{\T_{\text{BN}}}=I_{\T_{\text{dec}}}+I_{\text{blue}}$ where
\begin{equation}
I_{\text{blue}}=\ideal{p_{[v_3]}p_{[v_6]}-p_{[v_4]}p_{[v_5]}}
=\ideal{p_{00++}p_{11++}-p_{01++}p_{10++}}.
\end{equation}

Indeed, the full ideal $\ker \varphi $ of this tree is not necessarily toric because the condition \cref{eq:star} is not fulfilled: we can easily use the tree graph to check that $t(v_3)t(v_6)\not=t(v_4)t(v_5)$.

Using any computer algebra software, we find that the ideal of model invariants $I_{\T_{\text{BN}}}=\ker\varphi$ is equal to the kernel 
of the algebraic parametrisation. We expect this result because in this case $\ipaths=\impaths$ see \cref{conj:mpaths} in the 
discussion. The theory developed by \citet{Geiger.etal.2006} did not supply the means to find this algebraic characterisation, their study is restricted to decomposable models. However, here we study a Bayesian network on four binary random variables, so one of the subjects of the thorough analysis provided by \citet[Table~1, \#21]{Garcia.etal.2005}. Using prime decomposition of the conditional independence ideal, these authors find that the ideal of model invariants of this Bayesian network is prime, of codimension~7, of degree~32, and has 13 minimal generators. From
the tree we can readily obtain the same calculations for the codimension and the number of minimal generators of the model by using the dimension count in \cref{sub:dim} and the equations in \cref{eq:ipaths} respectively.
\medskip

In a third step, consider the context-specific Bayesian network represented by the graph in \cref{fig:bioBN} and with the extra condition that $X_3\independent X_2~|~X_1=0$ is true. In words, we now embed the information that the probability of survival of a cell in the culture does not depend on its activity, given that the environment was hostile. Whilst the directed acyclic graph cannot code this condition graphically, we can immediately read it from the green stage in the staged tree $\T$ given in \cref{fig:biotree}. The full ideal of model invariants $I_{\T}=I_{\T_{\text{BN}}}+I_{\text{green}}$ is now equal to
\begin{equation}
I_{\T}=I_{\text{red}}+I_{\text{yellow}}+I_{\text{blue}}+I_{\text{green}}.
\end{equation}
Here, the green path differences generate the ideal
\begin{equation}
I_{\text{green}}=\ideal{p_{[v_7]}p_{[v_{10}]}-p_{[v_9]}p_{[v_8]}}
=\ideal{p_{000+}p_{011+}-p_{010+}p_{001+}}
\end{equation}
in the obvious notation for marginal probabilities. This result extends the results of both of  \citet{Garcia.etal.2005,Geiger.etal.2006} who do not study context-specific Bayesian networks.

Using \cref{thm:toric}, we can see again that $\ker \varphi $ is not necessarily toric. This is because even though now $v_3$ and $v_4$ are in the same position, implying that $t(v_3)=t(v_4)$, we still have $t(v_5)\not=t(v_6)$. So $t(v_3)t(v_6)\not=t(v_4)t(v_5)$ and condition \cref{eq:star} is not fulfilled.

Because in the staged tree $\T$ tree path differences can be extended, we find that in this case the ideal of model invariants is not equal to the kernel of the algebraic parametrisation, $I_{\T}\not=\ker\varphi$. In fact $I_{\T}$ is not radical and has five associated primes, one of them being $\ker \varphi$.
We can computationally check in this case that the kernel of $\varphi$ is equal to the ideal of maximal paths: $\ker\varphi=\impaths$. This observation, together with all of the examples presented in this paper, is strong evidence for \cref{conj:mpaths} which we formulate in the discussion at the end of this paper.
\medskip

In a final step, we observe that many of the unfoldings in the graph of $\T$ as given in \cref{fig:biotree} are logically impossible. For instance, the measure of recovery for cells which have died is nonsensical. In the same fashion, the problem description states that within the cell culture, cells get damaged only if the surrounding environment was hostile. This implies that the measure of survival is nonsensical in benign environments. As a result, many of the atomic probabilities in the context-specific Bayesian network need to be assigned probability zero. We can avoid this redundancy by directly modelling the given situation using the staged tree depicted in \cref{fig:biostaged}. We denote this tree $\T_{>0}$ because all of its labels are strictly positive.  The dimension of the corresponding model $\M_{\T_{>0}}$ is much smaller: the model variety is of dimension four in seven-dimensional space, rather than of dimension six in sixteen-dimensional space, as was the case for $\treemodel$.

The ideal of model invariants for the staged tree $\T_{>0}$ can be calculated in exactly the same fashion as presented above, now resulting in
\begin{equation}
\begin{split}
I_{\T_{>0}}&=I'_{\text{yellow}}+I'_{\text{green}}+I'_{\text{blue}}\\
&=\ideal{p_2p_6-p_5p_3}
+\ideal{p_1(p_5+p_6)-p_4(p_2+p_3)}
+\ideal{(p_1+p_2+p_3)p_8-p_7(p_4+p_5+p_6)}
\end{split}
\end{equation}
where we read the atomic probabilities from top to bottom in \cref{fig:biostaged}.  Again, we can calculate that $I_{\T_{>0}}\not=\ker\varphi$ but that the kernel of the algebraic parametrisation is given by the maximal paths:
\begin{equation}
\impaths=\ker\varphi=\ideal{p_3p_5-p_2p_6, p_2p_4-p_1p_5, p_3p_4-p_1p_6,p_4p_7-p_1p_8, p_5p_7-p_2p_8, p_6p_7-p_3p_8}.
\end{equation}
These generators form a Gr\"obner basis with respect to the reverse lexicographical term order with $p_1>\cdots >p_8$.

\section{Discussion}
\label{sect:discussion}

Throughout this text we have analysed the algebraic and geometric properties of the ideal $I_\T$ of model invariants of $\M_\T$. We have seen that this ideal has a distinguished prime component, namely $\ker \varphi$, and we have fully characterised in \cref{thm:toric} conditions under which this ideal is toric. Although we have not always explicitly stated it, all of the examples we have seen  in this paper have the property that $\ker \varphi$ is equal to the ideal $\impaths$. This leads us to:

\begin{conjecture}\label{conj:mpaths}
Let $\T$ be a staged tree. Then the kernel of $\varphi$ is exactly the ideal generated by maximal paths, 
\begin{equation}
\ker \varphi = \impaths.
\end{equation}
\end{conjecture}
 
For brevity, we decided not to discuss the algorithmic properties of the  ideal $\impaths$ in this text. However, the available nested polynomial representations of staged trees, as in \cref{eq:tv} and analysed by \citet{Goergen.etal.2018}, provide a promising computational tool to find $\impaths$. This is because using these, we can recursively relate extensions of path differences to the polynomials $t(v)$ for any $v\in V$.  

Restricting our study to toric staged tree models in \cref{sub:thmstar}, we saw that the algebraic characterisation of these is closely related to condition \cref{eq:star} which is necessary for paths to extend. In the context of decomposable graphical models, we know that the ideal of model invariants given by conditional independence statements is the toric ideal defining the kernel of the associated parameterisation, and that moreover the generators of this ideal form a Gr\"obner basis. In the analysis conducted in \cref{sect:example} of this paper, we saw that for decomposable models the ideal $\impaths$ is exactly the ideal obtained by \citet{Geiger.etal.2006} and that therefore its generators form a Gr\"obner basis. In the toric model represented by the tree $\T_{>0}$ from \cref{fig:biostaged}, we can see that not only $\ker \varphi= \impaths$ but the binomial generators of $\impaths$ form a Gr\"obner bases of the ideal they generate. This leads is to state a variation of \cref{conj:mpaths} for toric staged tree models.

\begin{conjecture}
Let $\T=(V,E)$ be a staged tree and suppose that condition \cref{eq:star} holds for all vertices $v,w \in V$ which are in the same stage. Then 
\begin{equation}
\ker \varphi= \impaths
\end{equation}
and the generators of $\impaths$ corresponding to path differences of maximal extensions form a Gr\"{o}bner basis of $\ker \varphi$.
\end{conjecture}

\section*{Acknowledgement}
The authors would like to thank the anonymous reviewers of this paper for their careful reading and detailed comments that significantly improved the exposition of this paper.

Part of this research was supported through the programme \enquote{Oberwolfach Leibniz Fellows} by the Mathematisches Forschungsinstitut Oberwolfach in 2017 and by the Deutsche Forschungsgemeinschaft (DFG, German Research Foundation) -- 314838170, GRK 2297 MathCoRe.
All of the computations in this paper have been carried out using the freely available software \texttt{Macaulay2} \citep{M2}.

\bibliographystyle{elsarticle-harv}
\bibliography{complete_bibliography}
\end{document}